\documentclass[preprint]{elsarticle}

\makeatletter
\def\ps@pprintTitle{%
	\let\@oddhead\@empty
	\let\@evenhead\@empty
	\def\@oddfoot{\footnotesize\itshape
		{} \hfill\today}%
	\let\@evenfoot\@oddfoot
}
\makeatother

\usepackage{latexsym}
\usepackage{lineno}
\usepackage{hyperref}
\usepackage{indentfirst}
\usepackage{amsxtra}
\usepackage{amssymb}
\usepackage{amsthm}
\usepackage{natbib}
\usepackage{amsmath}
\usepackage{tikz}
\usepackage{amscd}
\usepackage{MnSymbol}
\usepackage[a4paper,left=1.5cm,right=1.5cm]{geometry}

\usepackage[capitalise]{cleveref}
\def\comment#1{{\color{red} #1}}

\def\setof#1#2{\{#1\,\mid\,#2\}}

\def\aff#1{\mathrm{Aff}(#1)}
\def\Z{\mathbb Z}

\def\D{\Delta}
\def\comment#1{{\color{red} #1}}
\def\setof#1#2{\{#1\, : \,#2\}}
\def\A{\mathcal A}

\newtheorem{theor}{Theorem}
\newtheorem{prop}[theor]{Proposition}
\newtheorem{cor}[theor]{Corollary}
\newtheorem{lemma}[theor]{Lemma}
\theoremstyle{definition} 
\newtheorem{defin}{Definition}
\newtheorem{rem}{Remark}
\newtheorem{rems}{Remarks}
\newtheorem{ques}{Question}

\newtheorem{ex}[theor]{Example}
\newtheorem{exs}[theor]{Examples}

\def\dis{\mathrm{Dis}}
\def\setof#1#2{\{#1\, : \,#2\}}

\DeclareMathOperator{\Sym}{Sym}
\DeclareMathOperator{\id}{id}
\DeclareMathOperator{\Aut}{Aut}
\DeclareMathOperator{\End}{End}

\begin{document}

\begin{frontmatter}
	\title{Involutive (simple) latin solutions of the Yang-Baxter equation and related (left) quasigroups
 }

  \author{Marco BONATTO}
  \ead{marco.bonatto.87@gmail.com }
	\tnotetext[mytitlenote]{The second author is a member of GNSAGA (INdAM). The second author was partially supported by the MAD project Cod. ARS$01\_00717$. Corresponding author: Marco Castelli.}
	\author{Marco CASTELLI}
	\ead{marco.castelli@unisalento.it - marcolmc88@gmail.com }

	%

\begin{abstract}
In this paper, we study involutive non-degenerate set-theoretic solutions of the Yang-Baxter equation with regular displacement group. In particular, we completely describe the blocks of imprimitivity and the congruences of the irretractable ones, that we show belonging to the class of the latin solutions. Among these solutions, we characterise the simple ones having nilpotent permutation group. A more precise description involving the First Weyl Algebra will be provided when the displacement group is abelian and normal in the total permutation group, and we enumerate and classify the simple ones having minimal size $p^p$, for an arbitrary prime number $p$. Finally, we illustrate our results by some examples.  

\end{abstract}
\begin{keyword}
\texttt{imprimitive group \sep quasigroup\sep Yang-Baxter equation\sep brace\sep cycle set }
\MSC[2020] 16T25\sep 81R50  \sep 20N05
\end{keyword}

\end{frontmatter}

\section*{Introduction}

The quantum Yang-Baxter equation first appeared in theoretical physics, in a paper by C.N. Yang \cite{yang1967}, and in statistical mechanics, in R.J. Baxter's work \cite{bax72}. On the other hand, it has been studied also been studied from a  mathematical viewpoint. In 1992 Drinfel'd \cite{drinfeld1992some} suggested the study of \emph{set-theoretical} solutions, i.e. solutions that act on a basis of the underlying vector space. Specifically, a \emph{set-theoretic solution of the Yang-Baxter equation} on a non-empty set $X$ is a pair $\left(X,r\right)$, where 
$r:X\times X\to X\times X$ is a map such that the relation
\begin{align*}
\left(r\times\id_X\right)
\left(\id_X\times r\right)
\left(r\times\id_X\right)
= 
\left(\id_X\times r\right)
\left(r\times\id_X\right)
\left(\id_X\times r\right)
\end{align*}
is satisfied. Set theoretical solutions provide solutions of the Yang-Baxter equation by linearization.
Writing a solution $(X,r)$ as $r\left(x,y\right) = \left(\lambda_x\left(y\right),\rho_y\left(x\right)\right)$, with
$\lambda_x, \rho_x$ maps from $X$ into itself, for every $x\in X$, we say that $(X, r)$ is \emph{non-degenerate} if $\lambda_x,\rho_x\in \Sym_X$, for every $x\in X$, and \emph{involutive} if $r^2=\id_{X\times X}$.\\
The seminal papers due by Gateva-Ivanova and Van Den Bergh  \cite{gateva1998semigroups} and Etingov, Schedler, and Soloviev \cite{etingof1998set} encouraged several authors to the study of the involutive non-degenerate set-theoretic solutions (which we simply call \emph{involutive solutions}). Recall that a solution $(X,r)$ is said to be \emph{decomposable} if $X$ can be expressed as a disjoint union of two non-empty subsets $Y_1$ and $Y_2$ such that $r(Y_i\times Y_i)\subseteq Y_i\times Y_i$ for all $i\in \{1,2\}$, and $(X,r)$ is called \emph{indecomposable}, otherwise. To attack the classification-problem of involutive solutions, as first attempt one can restrict the investigations to the class of indecomposable ones. This approach is motivated by the fact that indecomposable solutions are in some sense the fundamental blocks that allow to understand all the involutive ones (see \cite[Section $3$]{etingof1998set}). 
Even if we focus on indecomposable involutive solutions, a full classification is still hard, and in the last few years several authors investigated the so-called simple indecomposable involutive solutions, where a solution $(X,r)$ is said to be \emph{simple} if every epimorphic image $(Y,s)$ is isomorphic to $(X,r)$ or is a singleton. A notion of simple involutive solution (not necessarily indecomposable) was at first introduced by Vendramin in \cite{vendramin2016extensions} for involutive solutions not necessarily indecomposable. Even if his definition of simplicity is different, it coincides in the indecomposable case. Every simple involutive solution of size bigger than $2$ is indecomposable by \cite{cedo2021constructing}. These solutions play a special role since they are the “fundamental blocks" to construct all the indecomposable involutive ones by dynamical extensions (see \cite[Proposition 2]{cacsp2018} and \cite[Corollary 2.13]{vendramin2016extensions}).\\
Even if the construction of simple involutive solutions appeared quite hard, several interesting results were obtained in the last few years. If $p$ is an odd prime number there exist a unique simple involutive solution of size $p$, see \cite[Theorem 2.13]{etingof1998set}. For every natural number $n>1$, a simple involutive solution of cardinality $p^n$ was recently constructed in \cite{cedo2024simplep}, and in the case $n=2$ a full classification was given in \cite{dietzel2023indecomposable}. Further examples of cardinality $n^2$ and $mn^2$, where $n,m$ are arbitrary natural number, were constructed in \cite{cedo2021constructing,cedo2022new}. On the other hand, we know that an indecomposable involutive solutions of square-free cardinality can not be simple, as showed in \cite{cedo2022indecomposable}. 

In order to provide a theoretic description of simple solutions, in \cite{castelli2022characterization} the second author characterised them in terms of skew left braces provided by their associated permutation groups. Recall that a skew left brace, introduced in \cite{guarnieri2017skew,rump2007braces}, is a triple $(B,+,\circ)$ where $(B,+)$ and $(B,\circ)$ are groups and the equality
$$x\circ (y+z)=x\circ y-x+x\circ z $$
holds for all $x,y,z\in B$. The notion of simplicity was recently investigated in the general setting of bijective non-degenerate solutions (simply called \emph{solution}) not necessarily involutive. In \cite{castelli2022simplicity}, a criterion of simplicity was provided by blocks of imprimitivity of the associated permutation group. A remarkable generalization of \cite[Theorem $12$]{castelli2022characterization} was recently provided in \cite{colazzo2024simple}, where the simple bijective solutions (not necessarily involutive) were completely characterised by the underlying permutation skew left braces, and several examples of non-involutive simple solutions were provided. In this way, the main Theorem of \cite[Section $1$]{colazzo2024simple} represents a unifying tool to understand the notion of simplicity of several algebraic structures already present in literature.
By these last results, it is clear that for the simplicity of a solution the ideal $B^{(2)}$ of the permutation left brace plays a fundamental role. However, we are still far away from a concrete description of the maps $r$ related to simple solutions. Indeed, in \cite{colazzo2024simple} a complete description of simple solutions provided by skew left braces $B$ with $B^{(2)}=0$ or $B^{(3)}=0$ was given, but the remaining case $B^{(2)}=B^{(3)}\neq 0$ is still open. As noted in \cite[Section $1$]{colazzo2024simple}, the last case includes the involutive simple solutions that are not of Lyuabashenko type. To partially fill this gap, one of the main goals of this paper is the study of finite involutive simple solutions in which $B^{(2)}$, that in this case coincides with the so-called \emph{displacement group}, acts regularly on the underlying set. We will show that these solutions belong to the family of \emph{latin} solutions, i.e. solutions in which the maps $\lambda_x$ provide a quasigroup structure on the set $X$ (see \cite{bon2019}).
However, our investigation is not limited to simple solutions. Motivated by \cite[Theorem $3.1$]{cedo2020primitive}, where involutive solutions with primitive permutation group were completely classified, we focus on blocks of imprimitivity of involutive solutions (not necessarily simple). In addition to the main result of \cite{cedo2020primitive}, which ensures the existence of a complete blocks system, we will show that a complete blocks system can be always obtained by the orbits of a minimal normal subgroup of the permutation group. 
To develop a theory useful to simple involutive solutions, in the second part we specialize to the irretractable ones. By tools that come from quasigroups theory, we decribe \emph{all} the complete blocks system of irretractable involutive solutions with regular displacement group. Among these complete blocks systems, we distinguish the ones that give rise to congruences, providing an extension of a result involving the family of finite \emph{affine} quasigroups (see \cite[Lemma 4.8]{bon2019}) to a larger class of finite quasigroups. In order to understand more about congruences, using tools that come from the left quasigroups theory provided in \cite{semimedial}, we develop a covering theory of left quasigroups that allows to recover the one given in \cite{rump2023primes} for indecomposable involutive solutions as a special case. This is the reason for which, throught the paper, we will use the language of cycle sets. Recall that a \emph{cycle set} is a left quasigroup $(X,\cdotp)$ in which the equality $(x\cdotp y)\cdotp (x\cdotp z)=(y\cdotp x)\cdotp (y\cdotp z)$ holds for all $x,y,z\in X$. Cycle sets are useful to detect solutions of the Yang-Baxter equation: indeed, non-degenerate cycle sets bijectively correspond to non-degenerate involutive solutions, and by \cite[Propositions 1-2]{rump2005decomposition} the correspondence is given by 
\begin{equation}\label{corrisp}
    (X,\cdotp) \longrightarrow  (X,r)\quad with\quad  r(x,y):=(\sigma_x^{-1}(y),\sigma_x^{-1}(y)\cdot x).
\end{equation}
As one can expect, indecomposable (resp. simple) involutive solutions correspond to indecomposable (resp. simple) cycle sets. In the core of the paper, we focus on simple involutive solutions $(X,r)$ with nilpotent regular displacement group. We will show that, in this setting, $X$ must have a prime-power cardinality $p^n$ and the $p$-Sylow subgroup of the permutation group $\mathcal{G}(X)$ is normal. 
In this case we can reduce the classification to solutions for which $\mathcal{G}(X)$ is a $p$-group. A complete description of these solutions will be provided, while the ones having cyclic displacement group will be concretely classified. These results will be refined if, in addition, the displacement group is abelian and normal in the so-called \emph{total permutation group}. In this last case, we exhibit a connection
with irreducible representations of the First Weyl Algebra over a field with a prime number of elements. We remark that the First Weyl Algebra has already appeared in the study of involutive solutions, see \cite{bon2019,rump2020one}. These solutions has always size $p^{k}$, for a prime number $p$ and a natural number $k$ divisible by $p$. Among these solutions, the ones having size $p^p$ will be completely classified. Finally, we illustrate our main results with some examples, and we use some of them to answer in the affirmative to \cite[Question 7.6]{cedo2021constructing}.

\section{Left quasigroups}

\subsection{Basics}

A {\it left quasigroup} is an algebraic structure given by a non-empty set $X$ together with a binary operation $\cdot:X\times X\longrightarrow X$, such that the map $$\sigma_x:y\mapsto x\cdot y$$ is bijective for every $x\in X$. Left quasigroups can also be understood as algebraic structure with two binary operations, namely $\{\cdot,\backslash\}$ where $x\backslash y=\sigma_x^{-1}(y)$ for every $x,y\in X$. In the finite case, the two description are equivalent (they might differ in the infinite case), so for the scope of this paper we are using the definition with just one binary operation. 

The {\it associated permutation group} (also called {\it left multiplication group}) is defined as
$$\mathcal{G}(X)=\langle \sigma_x\mid x\in X\rangle\leq Sym(X).$$
If such a group is transitive on $X$ we say that $X$ is {\it indecomposable} (the same property is also known as {\it connectedness}).

The map $$\delta_x:y\mapsto y\cdot x$$ does not need to be bijective. If $\delta_x$ is bijective for every $x\in X$ we say that $X$ is {\it latin}. In this case we can define another binary operation as $x/y=\delta_y^{-1}(x)$ for every $x,y\in X$. The set $X$ endowed with the three binary operations $\{\cdot, \backslash,/\}$ is called {\it quasigroup}. If the underlying set is finite the algebraic properties of the left quasigroup and the quasigroups are the same (they might differ in the infinite case). For latin left quasigroups we can define another subgroup of the symmetric group, namely the \emph{total permutation group} as 
$$\mathcal{TG}(X)=\langle \sigma_x, \delta_x\mid  x\in X\rangle .$$

An equivalence relation on $X$ is said to be a \emph{congruence} if for every $x,y,x',y'\in X$ we have that if $x\sim x'$ and $y\sim y'$ then $x\cdotp y\sim x'\cdotp y'$ and $\sigma_x^{-1}(y)\sim \sigma_{x'}^{-1}(y')$. From now on, we indicate by $\sim_{1_X}$ the congruence in which all the elements belong to the same equivalence class and by $\sim_{0_X}$ the trivial congruence induced by singleton classes. Accordingly the operation $[x]\cdot [y]=[x\cdot y]$ is well-defined and the quotient set $X/\sim$ together with this operation is a left quasigroup.

    Let $X,Y$ be left quasigroups. A \emph{homomorphism} from $X$ to $Y$ is a map $p:X\rightarrow Y$ such that $p(x\cdotp y)=p(x)\cdotp p(y)$ for all $x,y\in X$. If $p$ is surjective, then it will be called \emph{epimorphism}, while a bijective homomorphism is said to be \emph{isomorphism}. An isomorphism from a left quasigroup $X$ to itself will be called \emph{automorphism} and the automorphism group will be indicated by $Aut(X,\cdotp)$.

Epimorphisms and congruence are essentially the same thing. Indeed if $\sim$ is a congruence, the canonical map $X\longrightarrow X/\sim$ is an epimorphism. On the other hand, the equivalence relation $\sim_p$ induced by $p$ given by $x\sim_p y \Longleftrightarrow p(x)=p(y)$ for all $x,y\in X$, is a congruence of $X$ and $Y\cong X/{\sim_p}$. 

\begin{lemma}\label{lemmaprep1}\cite[Lemma 1.8]{AG}
    If $X,Y$ are left quasigroups and $p:X\rightarrow Y$ is a left quasigroups epimorphism, the assignment $\sigma_x\mapsto \sigma_{p(x)}$ extends to a groups epimorphism $\bar{p}:\mathcal{G}(X)\rightarrow \mathcal{G}(Y)$. 
\end{lemma}

We can introduce a particular family of normal subgroups of the associated permutation group.

\begin{defin}\cite[]{semimedial}\label{admissiblesub}
    Let $X$ be a left quasigroup and $H$ be a normal subgroup of $\mathcal{G}(X)$. Then, $H$ is \emph{admissible} if $\sigma_x^{-1}\sigma_{h(x)}\in H$ for all $x\in X,$ $h\in H$. 
\end{defin}

The lattice of admissible subgroups is a sublattice of the lattice of the normal subgroups of $\mathcal{G}(X)$. Note that, if $H$ is admissible, then $H$ is normal and so we have that $\sigma_x^{-1}\sigma_{h(x)}\in H$ for all $x\in X,$ $h\in H$ if and only if $\sigma_x\sigma^{-1}_{h(x)}\in H$ for all $x\in X,$ $h\in H$

According to \cite{semimedial} we can define a Galois connection between the congruence lattice of a left quasigroup and the sublattice of the admissible subgroups. Given a congruence $\sim$ we can define 
$$\mathcal{G}^\sim=\setof{h\in \mathcal{G}(X)}{h(x)\, \sim\, x\, \text{ for every }x\in X}.$$ 
Note that $\mathcal{G}^\sim$ is nothing but the kernel of the induced group homomorphism from $\mathcal{G}(X)$ onto $\mathcal{G}(X/\sim)$ defined in Lemma \ref{lemmaprep1} and it is admissible according to \cite[Corollary 1.9]{semimedial}. On the other hand, given an admissible subgroup $H$, we can define the equivalence relation 
$\sim_{\mathcal{O}_H}$ given by 
$$x\sim_{\mathcal{O}_H} y\Longleftrightarrow y=h(x) \textit{ for some } h\in H,$$ namely the equivalence relation with blocks given by the orbits of $H$. Such equivalence is a congruence \cite[Lemma 1.8]{semimedial}. We usually denote the factor left quasigroup with respect to $\sim_{\mathcal{O}_H}$ just by $X/H$.


Moreover, given a congruence $\sim$ we can define the subgroup $\dis_\sim$ as the normal closure of $\setof{\sigma_x \sigma_y^{-1}}{x\, \sim\, y}$ in the group $\mathcal{G}(X)$. Such subgroups are admissible. We denote $\dis_{\sim_{1_X}}$ just by $\dis(X)$ and we refer to it as the {\it displacement group} of $X$. By \cite[Section $1$]{semimedial}, the quotient $\mathcal{G}(X)/\dis(X)$ is cyclic, and in particular $\mathcal{G}(X)=\dis(X)\langle \sigma_x\rangle$, for every $x\in X$.

We can define the binary relation $\sim_\sigma$ on a left quasigroup $X$ given by 
$$x\sim_\sigma y :\Longleftrightarrow \sigma_x = \sigma_y$$ 
for all $x,y\in X$. The relation $\sim_{\sigma}$ is not a congruence in general, A left quasigroup $X$ is called \emph{irretractable} if $\sim_{\sigma}=\sim_{0_X}$ (sometimes in the literature the left quasigroups with the same property are called faithful). 

Indecomposable left quasigroups can always be described in terms of left cosets.

\begin{prop}\label{rap}
    Let $X$ be an indecomposable left quasigroup, $x$ an element of $X$ and $H:=\mathcal{G}(X)_x$ the stabilizer of $x$. Then, $X$ is isomorphic to the left quasigroup on the left cosets $\mathcal{G}(X)/H$ given by $gH\cdotp hH:=\sigma_{g(x)}hH$ for all $g,h\in \mathcal{G}(X)$.
\end{prop}

\begin{proof}
    By a standard calculation, one obtain that $(\mathcal{G}(X)/H,\cdotp)$ is a left quasigroup and $i:\mathcal{G}(X)/H\rightarrow X,$ $gH\mapsto g(x)$ is a left quasigroups isomorphism.
\end{proof}

The following result follows by a standard calculation.

\begin{prop}\label{prel3}
    Let $X$ be a latin left quasigroup. Then, $X$ is irretractable and $\dis(X)$ is transitive. In particular, $X$ is indecomposable.
\end{prop}

\begin{proof}
Assume that $\sigma_x=\sigma_y$. Then $\delta_z(x)=x\cdotp z=y\cdotp z=\delta_z(y)$ for every $z\in X$. Thus $x=y$, i.e. $X$ is irretractable. 

Since $X$ is latin, for every $a,t,z\in X$ there exist $b\in X$ such that $b\cdotp t=a\cdotp z$, therefore $\sigma_{b}^{-1}\sigma_a(z)=t$, hence $\dis(X)$ is transitive on $X$.
\end{proof}

Given a left quasigroup $(X,\cdotp)$, we can define several operations that makes $X$ into a different left quasigroup.

\begin{defin}
    If $(X,\cdotp)$ is a left quasigroup and $\alpha\in Aut(X,\cdotp)$, we can define another operation $\cdotp_{\alpha}$ on $X$ by $x\cdotp_{\alpha} y:=\alpha(x\cdotp y)$ for every $x,y\in X$. The left quasigroup $(X,\cdotp_{\alpha})$ will be indicated by $X_\alpha$ and will be called \emph{deformation of $X$ by $\alpha$}.
\end{defin}

\subsection{Coverings of indecomposable left quasigroups}

In this section we develop a covering theory for indecomposable left quasigroup which is consistent with the one developed in \cite[Section 2]{rump2020} for indecomposable cycle sets.

\noindent By \cref{lemmaprep1}, the following definition makes sense. Clearly, it is consistent with \cite[Definition 1]{rump2020}.

\begin{defin}\label{cov1}
    Let $X,Y$ be indecomposable left quasigroups and $p$ an epimorphism from $X$ to $Y$. Then, $p$ is said to be a \emph{covering} if the induced epimorphism $\bar{p}$ from $\mathcal{G}(X)$ to $\mathcal{G}(Y)$ is an isomorphism.
\end{defin}

Note that if $p$ is a covering, then $\dis_{\sim_p}\leq \mathcal{G}^{\sim_p}(X)=1$. Therefore, 
\begin{equation}\label{covering are below lambda}
    p(x)=p(y) \text{ implies }\sigma_x=\sigma_y.
    \end{equation}
In the next result, we show that a left quasigroup $X$ always arises as epimorphic image of $\mathcal{G}(X)$ by a covering.

\begin{prop}\label{univ}
        Let $X$ be an indecomposable left quasigroup, $x$ an element of $X$ and $Y$ be the left quasigroup $(\mathcal{G}(X),\cdotp)$ given by $g\cdotp h:=\sigma_{g(x)}h$ for all $g,h\in \mathcal{G}(X)$. Then the map 
        $$u:Y\longrightarrow X,\quad g\mapsto g(x)$$ 
        is a epimorphism of left quasigroup. Moreover, $u$ is a covering. 
\end{prop}

\begin{proof}
    It follows by a standard calculation using that $X$ can be identified as the left coset left quasigroup as in \cref{rap}.
\end{proof}

 Following the terminology of \cite[Section 2]{rump2020} and \cite[Section 3]{rump2023primes}, $u$ will be called \emph{universal covering}. Now, we give the analog of \cite[Proposition 4]{rump2023primes} for arbitrary indecomposable left quasigroups.

\begin{prop}
    Let $X,Y$ be indecomposable left quasigroups and suppose that $p:Y\rightarrow X$ is a covering. Then, the universal covering $u:\mathcal{G}(X)\rightarrow X$ factors through $p$.
\end{prop}
\begin{proof}
Let $x\in X$ and $y\in Y $ such that $p(y)=x$ and $\bar{p}:\mathcal{G}(Y)\longrightarrow \mathcal{G}(X)$ the induced group homomorphism, that is invertible. The map
$$q:\mathcal{G}(X)\longrightarrow Y,\quad g\longrightarrow \bar{p}^{-1}(g)(y)$$
 is a left quasigroup epimorphism and $u=p\circ q$.
\end{proof}

The following theorem, which is the main result of the section, is an extension of \cite[Theorem 1]{rump2020} to indecomposable left quasigroups and allows to understand the structure of the epimorphisms. At first recall that, given two equivalence relations $\sim$ and $\sim'$, we say that $\sim$ is contained in $\sim'$ if $x\sim y$ implies $x\sim' y$ for all $x,y\in X$.

\begin{theor}\label{fattorizzleft}
            Let $X,Y$ be indecomposable left quasigroups and $p$ an epimorphism from $X$ to $Y$. Then, $p$ factorize as $qr$ where $r:X\rightarrow X/\mathcal{G}^{\sim_p}$ is the canonical projection and $q:X/\mathcal{G}^{\sim_p}\rightarrow Y$, where $q$ is a covering of indecomposable left quasigroups.
\end{theor}

\begin{proof}
    Let $\sim=\sim_{\mathcal{O}_{\mathcal{G}^{\sim_p}}}$. Thus, the map $p$ factors through the canonical map onto $X/\sim$, $$X\longrightarrow X/\sim\overset{q}{\longrightarrow} Y.$$
Since $\sim$ is contained in $\sim_p$ then $\mathcal{G}^\sim\leq \mathcal{G}^{\sim_p}$. On the other hand $\mathcal{G}^{\sim_p}\leq \mathcal{G}^\sim$ since the blocks of $\sim$ are the orbits of $\mathcal{G}^{\sim_p}$. Thus, $\mathcal{G}^\sim=\mathcal{G}^{\sim_p}$. Therefore $\mathcal{G}(X/\sim)=\mathcal{G}(X)/\mathcal{G}^{\sim}=\mathcal{G}(X)/\mathcal{G}^{\sim_p}=\mathcal{G}(Y)$. Therefore the map $q$ is a covering. 
\end{proof}

According to \eqref{covering are below lambda}, we obtain that in the irretractable case there are no non-trivial coverings.

\begin{prop}\label{nocov}
    Let $X,Y$ be left quasigroups and let $p:X\longrightarrow Y$ be a covering. If $X$ is irretractable then $p$ is an isomorphism.
\end{prop}
\begin{proof}
Note that $\sim_p$ is contained in $\sim_{\sigma}=\sim_{0_X}$. So $p$ is injective.
\end{proof}

\subsection{Blocks of imprimitivity}

Let $G$ be a group acting on a set $X$ and $\sim$ be an equivalence relation of $X$. We say that $\sim$ is $G$-invariant if $g(x)\,\sim g(y) $ provided $x\, \sim\, y$. If $G$ is transitive, a subset $Y\subseteq X$ is called a {\it block of imprimitivity} if $g(Y)\cap Y$ is either empty or equal to $Y$, for every $g\in G$, and the partition $P$ of $X$ given by $P:=\{g(Y)\}_{g\in G}$ is said to be a \emph{complete blocks system}. If the action of $G$ has no non-trivial blocks of imprimitivity it is said to be {\it primitive}. If the action of $G$ is primitive we say that $X$ is {\it primitive}.

Let $X$ be a set and $H$ be a group of permutations on the set $X$. We define
\begin{eqnarray}
    &x\, \sim_H\, y \,\Longleftrightarrow \, H_x=H_y.\label{sim H}\\
    & x\, \sim_{\mathcal{O}_H}\, y \,\Longleftrightarrow \, x=h(y), \text{ for some }h\in H.\label{relut}
\end{eqnarray}
Clearly $\sim_H$ and $\sim_{\mathcal{O}_H}$ are equivalence relations.

Let us prove that the previous equivalence relations provide invariant partitions under further assumptions.

\begin{lemma}\label{orbits of a normal}
    Let $G$ be a group acting on a set $X$ and $H\trianglelefteq G$. Then $\sim_{\mathcal{O}_H}$ and $\sim_H$ are $G$-invariant equivalence relation on $X$.  
\end{lemma}

\begin{proof}
    Let $g\in G$, $h\in H$ and $x\in X$. Then $gh(x)=ghg^{-1} g(x)$. Then $g(x)$ and $gh(x)$ are in the same orbit w.r.t. $G$.

    If $H_x=H_y$ then $H_{g(x)}=g H_x g^{-1}=g H_y g^{-1}=H_{g(y)}$, since $H$ is normal in $G$.
\end{proof}


\begin{cor}   \label{stabilizers for normal}
Let $X$ be a left quasigroup and $H$ be a normal subgroup of $\mathcal{G}(X)$. Then $\sim_H$ is a $\mathcal{G}(X)$-invariant equivalence relation.
\end{cor}

\begin{proof}
We can apply Lemma \ref{orbits of a normal}.
\end{proof}

Note that Corollary \ref{stabilizers for normal} applies to admissible subgroups.

\begin{lemma}
Let $X$ be a left quasigroup and $H$ be a normal subgroup of $\mathcal{G}(X)$.  
\begin{itemize}
    \item[(i)] If $\sim_H=\sim_{0_X}$ then $N_{\mathcal{G}(X)}(H_x)\leq \mathcal{G}(X)_x$ for every $x\in X$. In particular $|Z(\mathcal{G}(X))|=1$.

\item[(ii)] $\sim_H=\sim_{1_X}$ if and only if $|H_x|=1$ for every $x\in X$, i.e. $H$ is semiregular.

\end{itemize}
\end{lemma}
\begin{proof}
Let $G:=\mathcal{G}(X)$.

    (i) Let $g\in N_{G}(G_x)$. Then $H_x=g H_x g^{-1}=H_{g(x)}$. Then $x\,\sim_{H}\, g(x)$ and so $g(x)=x$. Namely, $N_{G}(G_x)\leq G_x$. Since $Z(G)\leq C_G(H_x)$ for every $x\in X$, then $Z(G)\leq \cap_{x\in X} H_x=<id_X>$.

    (ii) Let $g\in H_x$. If $x\, \sim_H\, y$ for every $x,y\in X$ then $g(y)=y$ for every $y\in X$, i.e. $g=1$.
\end{proof}

\begin{cor}
Let $X$ be a indecomposable left quasigroup and $H$ be a normal subgroup of $\mathcal{G}(X)$. Then $\sim_H=\sim_{0_X}$ if and only if $N_{\mathcal{G}(X)}(H_x)\leq \mathcal{G}(X)_x$ for every $x\in X$. 
\end{cor}

\begin{proof}
    The stabilizers are all conjugate since for every $x,y\in X$ there exists $g\in \mathcal{G}(X)$ such that $y=g(x)$. Therefore $H_y=H_{g(x)}=g H_x g^{-1}=H_x$ if and only if $g\in N_{\mathcal{G}(X)}(H_x)$. Therefore, $\sim_H=\sim_{0_X}$ if and only if $N_{\mathcal{G}(X)}(H_x)\leq \mathcal{G}(X)_x$ for every $x\in X$.
\end{proof}

Note that congruences of a left quasigroup $X$ are $\mathcal{G}(X)$-invariant. We show that the converse does not hold in general by using the equivalence relations defined in this section. 

\begin{ex}
    Let $X:=\{1,2,3,4\}$ be the left quasigroup given by $\sigma_1:=(2,4)$, $\sigma_2:=(1,2,3,4)$, $\sigma_3:=(1,4,3,2)$, and $\sigma_4:=(1,3)$. Then, $\mathcal{G}(X)$ is isomorphic to the dihedral group of size $8$ and the subgroup $H$ generated by the set  $\{(1,3)(2,4),(2,4)\}$ is a normal subgroup of $\mathcal{G}(X)$. The relations $\sim_H$ and $\sim_{\mathcal{O}_H}$ coincide and are given by the partition $\{\{1,3 \},\{2,4 \} \}$. Such relations are not congruences of $X$.
\end{ex}




\section{Left braces}

Following \cite[Definition 1]{cedo2014braces}, a set $B$ endowed of two operations $+$ and $\circ$ is said to be a \textit{left brace} if $(B,+)$ is an abelian group, $(B,\circ)$ a group, and
$$
    x\circ (y + z) + x
    = x\circ y + x\circ z,
$$
for all $x,y,z\in B$. The group $(B,+)$ will be called the \emph{additive} group of the left brace, while the group $(B,\circ)$ will be called the \emph{multiplicative group.}

\begin{exs}\label{esbrace}
\begin{itemize}
    \item[1)] If $(B,+)$ is an abelian group, then the operation $\circ$ given by $x\circ y:=x+y$ give rise to a left brace which we will call \emph{trivial}.
\item[2)] Let $B:=(\mathbb{Z}/p^2\mathbb{Z},+)$ and $\circ$ be the binary operation on $B$ given by $x\circ y:=x+y+p\cdotp x\cdotp y$ (where $\cdotp$ is the ring-multiplication of $\mathbb{Z}/p^2\mathbb{Z}$). Then, $(B,+,\circ)$ is a left brace.\end{itemize}

\end{exs}

If $(B_1,+,\circ)$ and $(B_2,+',\circ')$ are left braces, a homomorphism $\psi$ between $B_1$ and $B_2$ is a function from $B_1$ to $B_2$ such that $\psi(x+y)=\psi(x)+'\psi(y)$ and $\psi(x\circ y)=\psi(x)\circ'\psi(y)$, for all $x,y\in B_1$. 

\smallskip

Given a left brace $B$ and $x\in B$, let us denote by $\lambda_x:B\longrightarrow B$ the map from $B$ into itself defined by
\begin{equation}\label{eq:gamma}
    \lambda_x(y):= - x + x\circ y,
\end{equation} 
for all $y\in B$. Let us recall the properties of the maps $\lambda_x$.
\begin{prop}\label{action}\cite[Proposition 2]{rump2007braces},\cite[Lemma 1]{cedo2014braces}
Let $B$ be a left brace. Then, the following are satisfied: 
\begin{itemize}
\item[1)] $\lambda_x\in\Aut(B,+)$, for every $x\in B$;
\item[2)] the map $\lambda:B\longrightarrow \Aut(B,+)$, $x\mapsto \lambda_x$ is a group homomorphism from $(B,\circ)$ into $\Aut(B,+)$.
\end{itemize}

\end{prop}



For the following definition, we refer the reader to \cite[pg. 160]{rump2007braces} and \cite[Definition 3]{cedo2014braces}.

\begin{defin}
Let $B$ be a left brace. A subset $I$ of $B$ is said to be a \textit{left ideal} if it is a subgroup of the multiplicative group and $\lambda_x(I)\subseteq I$, for every $x\in B$. Moreover, a left ideal is an \textit{ideal} if it is a normal subgroup of the multiplicative group.
\end{defin}
%

If $I$ is an ideal of a left brace $B$, the quotient set $B/I$ has a canonical left brace structure called the \emph{quotient left brace} of $B$ modulo $I$ (the operations are defined as $(x+I)+(y+I)=x+y+I$ and $(x+I)\circ (y+I)=(x\circ y)+I$ for every $x,y\in B$). Ideals are nothing other than kernels of homomorphisms, where the kernel of a homomorphism $\psi$ is the set given by $\ker(\psi):=\{x\in B_1|\psi(x)=0\}$. Indeed, the kernel of a homomorphism is always an ideal of a left brace, and conversely an ideal $I$ is the kernel of the canonical projection $B\longrightarrow B/I$.  

The set $\{0\}$ is an ideal and it will be called the \emph{trivial} ideal. A left brace $B$ which contains no ideals different from $\{0\}$ and $B$ will be called a \emph{simple} left brace.

\noindent A standard ideal of a left brace $B$ is the \emph{socle}, indicated by $Soc(B)$ and given by the kernel of the map $\lambda$. Another useful ideal of a left brace $B$, given in \cite[Corollary of Proposition 6]{rump2007braces} and indicated by $B^2$, is the one given by the additive subgroup generated by the set $\{x*y\hspace{1mm} |\hspace{1mm} x,y\in B \}$, where $x*y:=-x+x\circ y-y$ for all $x,y\in B$. If $I$ is an ideal of a left brace $B$, then $I*B$ and $B*I$ are subset of $I$, where $I*B$ and $B*I$ are defined in the same way of $B^2$.\\

In general, if $B$ is a left brace and $x\in B$ the map $\lambda_x$ does not belong to $Aut(B,+,\circ)$. The following lemma, that contains the same idea of \cite[Lemma $1.1$]{dietzel2023indecomposable}, shows that under additional hypothesis, a map $\lambda_x$ induces an automorphism on some left ideals.

\begin{lemma}\label{lemaut}
    Let $B$ be a left brace and $C,D$ two left ideals of $B$. Moreover, suppose that $\lambda_x(y)=y$ for all $x\in C+ D$ and $y\in D$. Then, the restriction of $\lambda_y$ to $C$ is an element of $ Aut(C,+,\circ)$.
\end{lemma}


\begin{proof}
Let $y\in D$ and $z\in C.$ Then we have
\begin{eqnarray}
   & & \lambda_y(z)=-y+y\circ z=-y+y\circ z\circ y^- \circ y= \nonumber \\
   &=& -y+y\circ z\circ y^- - y\circ z \circ y^-+ y\circ z\circ y^- \circ y  \nonumber \\
   &=& -y+y\circ z \circ y^-+\lambda_{y\circ z \circ y^-}(y) \nonumber \\
   &=& y\circ z \circ y^- \nonumber
\end{eqnarray}
   therefore the restriction of $\lambda_y$ to $C$ is an element of $Aut(C,+,\circ)$. 
\end{proof}
 Recall that every Sylow subgroup of the additive group is always a left ideal. Using the left ideal structures, one can show that if $B_1$,...,$B_t$ are Sylow subgroups of $(B,+)$, then the subgroup $B_1+...+B_t$ of the additive group is equal to $B_1\circ ... \circ B_t$ and is also a subgroup of the multiplicative group. 

\begin{lemma}(Proposition $3$, \cite{rump2019classification} and Proposition $2.6$, \cite{cedo2020primitive})\label{preparteo}
 Let $B$ be a left brace and $B_{p_1},...,B_{p_r}$ be the Sylow subgroups of $(B,+)$. Then:
 \begin{itemize}
     \item[1)] if $N$ is a normal subgroup of $(B,\circ)$ contained in $B_{p_1}$, then $\lambda_n(b)=b$ for all $n\in N$ and $b\in B_{p_2}\circ...\circ B_{p_r}$;
     \item[2)] $B_{p_1}$ is a normal subgroup of $(B,\circ)$ if and only if $\lambda_g(b)=b$ for all $g\in B_{p_1}$ and $b\in B_{p_2}\circ...\circ B_{p_r}$.
 \end{itemize} 
\end{lemma}

\section{Cycle sets}

\subsection{Basics}
Following \cite{rump2005decomposition}, a left quasigroup $(X,\cdot)$ is said to be a \emph{cycle set} if the equality 
\begin{equation}\label{eqcycleset}
    (x\cdot y)\cdot (x\cdot z)=(y\cdot x)\cdot (y\cdot z),
\end{equation}
holds for all $x,y,z\in X$.  Moreover, a cycle set $(X,\cdot)$ is called \textit{non-degenerate} if the squaring map $\mathfrak{q}:X\longrightarrow X$, $x\mapsto x\cdot x$ is bijective. A cycle set will be called irretractable/indecomposable/latin if its underlying left quasigroup is irretractable/indecomposable/latin.

\begin{exs}\text{ }
\begin{itemize}
 \item[1)] If $X$ is a nonempty set and $\gamma \in Sym(X)$, the binary operation given by $x\cdotp y:=\gamma(y)$ makes $X$ into a non-degenerate cycle set. We will refer to these cycle sets as the \emph{trivial} cycle set provided by the permutation $\gamma$.  
\item[2)] If $X$ is a cycle set, then every deformation of $X$ is again a cycle set (see \cite[Section $3$]{Ru220}).
\end{itemize}
\end{exs}

\begin{rems}\text{ }
\begin{itemize}
    \item[1)] From now on, every cycle set will be non-degenerate. All the results will be given in the language of cycle sets, but they can be translated in terms of solutions by \cref{corrisp}.
    \item[2)]  A notion of displacement group for cycle sets was given in \cite[Section $2$]{bon2019}. Even if it is slightly different from the one given in Section $1$ for arbitrary left quasigroup, by \cite[Lemma $2.15$ and Proposition $2.17$]{bon2019} these two definitions coincide for non-degenerate cycle sets.
\end{itemize} 
\end{rems}

Let $B$ be a left brace. A \emph{cycle base} is a subset $X$ of $B$ that is a union of orbits with respect to the $\lambda$-action and generating the additive group $(B,+)$. If a cycle base is a single orbit, then is called \emph{transitive cycle base}. By \cite[Theorem 3]{rump2020}, every indecomposable cycle sets can be constructed by transitive cycle bases of left braces; on the other hand, every cycle set $X$ gives rise to a left brace structure on $\mathcal{G}(X)$ having as $\circ$ operation the usual maps composition (see Section $1$ of \cite{rump2020}). 
\begin{prop}[Theorem 3, \cite{rump2020}]\label{cosmod}
Let $(B,+,\circ)$ be a left brace, $Y\subset B$ a transitive cycle base, $a\in Y$,
and $K$ a core-free subgroup of $(B,\circ)$, contained in the
stabilizer  $B_{a}$ of $a$ (respect to the action $\lambda$). Then, the pair $(X,\cdotp)$ given by $X:=B/K$ and $\sigma_{x\circ K}(y\circ K):=\lambda_x(a)^-\circ y \circ K$ give rise to an indecomposable cycle set with $ \mathcal{G}(X)\cong B$.\\
 Conversely, every indecomposable cycle set $(X,\cdotp)$ can be obtained in this way, taking $B:=\mathcal{G}(X)$.
\end{prop}
\noindent The following result, implicitly contained in \cite{rump2020}, allows to avoid the construction of the previous proposition to recover indecomposable and irretractable cycle sets.
\begin{prop}[proof of Theorem 3, \cite{rump2020}]\label{idtr}
    Let $X$ be an indecomposable and irretractable cycle set. Then, there exist a transitive cycle base $Z$ of $\mathcal{G}(X)$ such that $X$ is isomorphic, as cycle set, to the cycle set on $Z$ given by $x\cdotp y:=\lambda_x^{-1}(y)$ for all $x,y\in Z$.
\end{prop}

By \cite[Section 3]{rump2023primes}, we know that if $X$ is an indecomposable cycle set and $I$ is an ideal of the left brace $\mathcal{G}(X)$, the $I$-orbits on $X$ give rise to a congruence of the cycle set $X$. This result can be recovered by looking at the connection between ideals and admissible subgroups. 

\begin{prop}\label{idadm}
    Let $X$ be an indecomposable cycle set and $H$ be an ideal of $\mathcal{G}(X)$. Then $H$ is an admissible subgroup. If $X$ is finite, the converse also follows.
\end{prop}

\begin{proof}
Let $H$ be an ideal of the left brace $\mathcal{G}(X)$. Then $H$ is a normal subgroup of the multiplicative group. Moreover, 
\begin{eqnarray}
& &\sigma_x\circ \sigma^{-1}_{h(x)}=\sigma_x\circ\lambda_h(\sigma_x^{-1})=\sigma_x\circ h\circ(h^{-1}+\sigma_x^{-1}) \nonumber \\
&=& \sigma_x\circ h\circ \sigma_x^{-1} \circ\sigma_x \circ(h^{-1}+\sigma_x^{-1}) \nonumber \\
&=& \sigma_x\circ h\circ \sigma_x^{-1}\circ \lambda_{\sigma_x}(h) \nonumber
\end{eqnarray}
for all $x\in X, h\in H$. Since $\sigma_x\circ h\circ \sigma_x^{-1} \in H $ and $\lambda_{\sigma_x}(h)\in H $ for all $x\in X, h\in H$, we have that $H$ is an admissible subgroup of $\mathcal{G}(X)$. 

If $X$ is finite and $H$ is an admissible subgroup, as in the previous implication we obtain $\sigma_x\circ \sigma^{-1}_{h(x)}=  \sigma_x\circ h\circ \sigma_x^{-1}\circ \lambda_{\sigma_x}(h)$ for all $x\in X$, $h\in H$. Since $H$ is a normal subgroup, we have that $\lambda_{\sigma_x}(h) $ for all $x\in X,h\in H$, and by finiteness of $X$, since the $\sigma_x$ generates $\mathcal{G}(X)$, we have that $\lambda_g(h)\in H$ for all $g\in \mathcal{G}(X),h\in H$. Hence $H$ is an ideal of $\mathcal{G}(X)$.
\end{proof}




\smallskip



\begin{cor}[Section 3, \cite{rump2023primes}]\label{epiid} 
     Let $X$ be a cycle set and $I$ an ideal of $\mathcal{G}(X)$. Then $\sim_{\mathcal{O}_I}$ is a congruence of $X$. 
     %
\end{cor}
\begin{proof}
The statement follows by \cref{idadm} and \cite[Lemma 1.8]{semimedial}.
\end{proof}

\subsection{Indecomposable cycle sets of prime-power size }

In this section we consider irretractable and indecomposable cycle sets having prime-power order. Under an additional hypothesis, these cycle sets can be obtained by irretractable and indecomposable cycle sets with a $p$-group permutation group. 

\smallskip

\begin{theor}\label{defpri}
        Let $X$ be an indecomposable and irretractable cycle set of size $p^n$ for some prime number $p$. Let $B_p$ the $p$-Sylow subgroup of $(\mathcal{G}(X),+)$. The following are equivalent:
        \begin{enumerate}
            \item     $B_p$ is a normal subgroup of $(\mathcal{G}(X),\circ) $.
            \item $X$ is a deformation $Y_\alpha$ of a cycle set $Y$ such that $\mathcal{G}(Y)\cong B_p$, $\alpha$ has order coprime with $p$ and $\alpha$ fixes an element $y\in Y$. 
        \end{enumerate}

\end{theor}

\begin{proof}
Assume that $B_{p_1}:=B_p$ is a normal subgroup of $(\mathcal{G}(X),\circ) $ and let $B_{p_2},...,B_{p_r}$ be the Sylow subgroup of $(\mathcal{G}(X),\circ)$ different from $B_p$. If $B_p=\mathcal{G}(X)$, we can take $\alpha=id_X$ to obtain the statement. Let $B_p<\mathcal{G}(X)$. Since $X$ is irretractable, we can identify $X$ with a transitive cycle base of $\mathcal{G}(X)$.  By orbit-stabilizer equation, if $St(x)$ is the stabilizer of $x$ respect to the action $\lambda$, we have $|\mathcal{G}(X)|=p^n\cdotp |St(x)|$. Since $\mathcal{G}(X)$ is solvable and finite, we have that the subgroup $B_{p_2}\circ ... \circ B_{p_r}$ is conjugate to a subgroup of $ St(x)$ so, up to replace $x$, we can suppose the inclusion $B_{p_2}\circ ... \circ B_{p_r}\subseteq St(x)$. Now, let $x_1,...,x_r\in \mathcal{G}(X)$, with $x_i\in B_{p_i}$ for every $i\in \{1,...,r\}$, such that $x=x_1+...+x_r$. Since this decomposition is unique and the subgroups $B_{p_i}$ are left ideals, it follows that $\lambda_{a}(x_i)=x_i$ for all $a\in B_{p_2}\circ ... \circ B_{p_r}$ and $i\in \{1,...,r\}$. Moreover, being $B_p$ normal in $\mathcal{G}(X)$, by \cref{preparteo} we have $\lambda_e(x_i)=x_i$ for all $e\in B_p$ and $i\in \{2,...,r\}$. From these facts, it follows that $\lambda_g(x_i)=x_i$ for all $g\in \mathcal{G}(X)$ and $i\in \{2,...,r\}$. Since $X$ is identified with a  $\lambda$-orbit, if $y,z$ are elements of $X$, there exist $y_1,z_1\in B_{p_1}$ such that $y=y_1+t$ and $z=z_1+t$, where $t=x_2+...+x_r$. Then
$$y\cdotp z=\lambda_{y_1+t}^{-1}(z_1+t)=\lambda^{-1}_{\lambda^{-1}_{y_1}(t)}(\lambda^{-1}_{y_1}(z_1))+\lambda^{-1}_{\lambda^{-1}_{y_1}(t)}(\lambda^{-1}_{y_1}(t))$$
and since $\lambda_g(t)=t$ for all $g\in \mathcal{G}(X)$, we obtain 
\begin{equation}
    y\cdotp z=\lambda_t^{-1}(\lambda^{-1}_{y_1}(z_1))+t.
\end{equation}
Therefore, we have that $X$ is the cycle set given by the lambda orbit $Y$ of $x_1$ respect to $B_p$ composed with $\lambda_{t}^{-1}$. By \cref{lemaut}, we obtain easily that the restriction of  $\lambda_{t}^{-1}$ to $Y$ is a cycle set automorphism, hence $X$ is a deformation of $Y$ by $\lambda_t^{-1}$, which has order coprime with $p$ and fixes $x_1$, and $\mathcal{G}(Y)\cong B_p$.\\
The converse follows by \cite[Lemma 4.5]{dietzel2023indecomposable}.
\end{proof}

\begin{rem}
    The previous result reduces the description of all indecomposable and irretractable cycle sets of prime-power size $p^n$ and with normal Sylow $p$-subgroup permutation group to the ones having a $p$-group permutation group. This is consistent with the first part of \cite[Section $4$]{dietzel2023indecomposable}. However, we highlight that \cref{defpri} is not useful to describe all the irretractable and indecomposable cycle set of size $p^n$. Indeed, Example $5$ of \cite{rump2023primes} is an irretractable and indecomposable cycle set of size $8$ and with permutation group isomorphic to the symmetric group $Sym(4)$. 
\end{rem}

\noindent Now, we can determine the structure of the permutation left braces associated to the cycle sets of \cref{defpri}. Recall that in a left brace $B$ the Sylow subgroups of $(B,+)$ coincide with the ones of $(B,\circ)$, and moreover if $B_1,...,B_t$ are Sylow subgroup of $(B,+)$, then $B_1+...+B_t=B_1\circ...\circ B_t$. Now, if $X$ is an indecomposable and irretractable cycle set of prime-power size $p^n$, we indicate by $B_{p_1}:=B_p$, $B_{p_2},...,B_{p_r}$ the Sylow subgroups of $\mathcal{G}(X)$.

\begin{prop}\label{struttbr}
    If $B_{p_1}$ is normal in $\mathcal{G}(X)$, then $\mathcal{G}(X)$ is isomorphic, as left brace, to the left braces semidirect product $B_{p_1}\rtimes_{\alpha} (B_{p_2}+...+B_{p_r})$, where $B_{p_2}+...+B_{p_r} $ is a trivial left brace with cyclic additive group and $\alpha$ correspond to the action by the map $\lambda$.
\end{prop}

\begin{proof}
    It follows by \cref{defpri} and \cite[Lemma 4.5]{dietzel2023indecomposable}.
\end{proof}

As a corollary, we recover \cite[Theorem 2.13]{etingof1998set}, that classify indecomposable cycle sets of prime size. 

\begin{cor}(Theorem 2.13, \cite{etingof1998set})\label{classp}
    Let $X$ be an indecomposable cycle set of prime size $p$. Then $X$ is isomorphic to the trivial cycle set of size $p$ given by an arbitrary $p$-cycle.
\end{cor}
\begin{proof}
    Suppose that $X$ is irretractable. Since by \cite[Theorem $2.15$]{etingof1998set} $\mathcal{G}(X)$ is solvable, by \cite[Exercise $7.2.12$]{robinson2012course} we have that it is isomorphic to a subgroup of the form $\mathbb{Z}/p\mathbb{Z}\rtimes N$, with $N$ a cyclic subgroup having order dividing $p-1$. Therefore, by \cref{defpri} $X$ is a deformation of an indecomposable cycle set $\bar{X}$ with permutation group isomorphic to $\mathbb{Z}/p\mathbb{Z}$. Since $X$ is irretractable, it follows that  $\bar{X}$ is irretractable, but this contradicts \cite[Proposition 1]{capiru2020}. Hence, we have that $X$ is retractable. Therefore by \cite[Lemma 1]{cacsp2018} $X$ is a cycle set given by $x\cdotp y:=\alpha(y)$ for all $x,y\in X$, where $\alpha$ is a $p$-cycle. By a standard calculation, one can show that the isomorphism class of this cycle set does not depend on the choice of the $p$-cycle.
\end{proof}

\subsection{Blocks of imprimitivity}

We turn our attention on the \emph{primitive} cycle sets, i.e. cycle sets having primitive permutation group. In this subsection, we recover the classification of these cycle sets given in \cite{cedo2020primitive}, and we provide a forward improvement that give us a method to find a non-trivial complete blocks system.

\smallskip

At first, recall that a permutation group $G$ of a finite set $X$ is said to be \emph{quasiprimitive} if every normal subgroup acts transitively on $X$ and it is called \emph{innately transitive} if there exist a minimal normal subgroup that acts transitively on $X$. Clearly, the implications ``primitive $\Rightarrow$ quasiprimitive" and ``quasiprimitive $\Rightarrow$ innately transitive" hold.
For solvable groups, these notions are equivalent (for more details, see \cite[Section 1]{bereczky2008groups} and \cite{praeger1993nan}). 

\begin{lemma}\label{inntr}
    Let $G$ be {a solvable permutation group of a finite set $X$}. Then, the following statements are equivalent:
    \begin{itemize}
        \item[1)] $G$ acts primitively on $X$;
        \item[2)] $G$ acts quasiprimitively on $X$;
        \item[3)] the action of $G$ on $X$ is innately transitive.
    \end{itemize}
\end{lemma}

\begin{prop}\label{quasipr}
    Let $X$ be an indecomposable cycle set with innately transitive permutation group. Then, $X$ is a trivial indecomposable cycle set of size $p$, for some prime number $p$.
\end{prop}


\begin{proof}
    By \cite[Theorem 2.15]{etingof1998set}, $\mathcal{G}(X)$ is a finite solvable group, hence by \cref{inntr} we have that $\mathcal{G}(X)$ acts primitively on $X$. Again by solvability of $\mathcal{G}(X)$ and by \cite[Theorem 7.2.6]{robinson2012course} we have that $|X|=p^n$ for some natural number $n$ and $\mathcal{G}(X)$ is isomorphic to  the semidirect product $N\rtimes G$ where $N$ is an elementary abelian $p$-group that acts transitively on $X$ and $G$ is the stabilizer of an element $x\in X$. Moreover, $N$ is a minimal normal subgroup of $\mathcal{G}(X)$ and is self-centralizing. By primitivity of $\mathcal{G}(X)$, either $X$ is irretractable or $n$ is equal to $1$. Suppose that $X$ is irretractable. Since $\mathcal{G}(X)$ is solvable and finite, we can suppose that $G$ contains $B_{p_2}+...+B_{p_r}$ and by \cref{idtr}, we have that $X$ can be identified with a transitive cycle base, therefore $\lambda_a(x)=x$ for all $a\in G$ and, if $x_1\in B_{p_1}$,...,$x_r\in B_{p_r}$ are such that $x=x_1+...+x_{r}$, it follows that $\lambda_a(x_i)=x_i$ for all $i\in \{1,...,r\}$. Moreover, since $N$ is a normal $p$-group, by \cref{preparteo} it follows that $\lambda_n(x_i)=x_i$ for all $i\in\{2,...r\}$. By these facts, we obtain that $B_{p_2}+...+B_{p_r}$ is a trivial left brace with cyclic additive group, and that $B_p$ acts trivially on $B_{p_2}+...+B_{p_r}$ by the map $\lambda$, therefore by \cref{preparteo} $B_p$ is normal in $\mathcal{G}(X)$. By minimality and self-centrality of $N$, it follows that $N=Z(B_p)=B_p$. By \cref{defpri}, $X$ is a deformation of an irretractable cycle set $\bar{X}$ with abelian permutation group isomorphic to $N$, but this contradicts \cite[Proposition 1]{capiru2020}. Therefore, $X$ has prime size $p$. The rest of the statement follows by \cref{classp}.
\end{proof}

   As we said before, \cref{quasipr} is essentially the main result of \cite{cedo2020primitive} with a forward improvement. The main difference is that now we have a concrete method to find at least a complete blocks system by the relation given in \cref{relut}.

\begin{cor}
    Let $X$ be a cycle set. Then, $X$ is primitive if and only if it is isomorphic to a cycle set of prime size $p$ given by $(\mathbb{Z}/p\mathbb{Z},\cdotp)$ and $x\cdotp y:=y+1$ for all $x,y\in \mathbb{Z}/p\mathbb{Z}$
\end{cor}

\section{Cycle sets with regular displacement group}
\subsection{Right linear cycle sets}

Let us introduce a family of left quasigroups constructed over groups.

\begin{defin}\label{defrightlin}
    A left quasigroup $X$ is said to be \emph{right linear} over a group $(G,\circ)$ if $X=G$ and there exist a map $\phi:G\longrightarrow G$ and $\psi\in Aut(G,\circ)$ such that $x\cdotp y:=\phi(x)\circ \psi(y)$ for all $x,y\in X$. If in addition $\phi\in Sym_G$, $X$ is latin.\\
    If $(G,\circ)$ is abelian and $x\cdot y=\phi(x)+\psi(y)+c$ for $\phi\in \End(G,+)$ and $c\in G$ we say that $(G,\cdotp)$ is {\it affine}. We denote such a left quasigroup by $\aff{G,\phi,\psi,c}$.

\end{defin}

By a standard calculation, one can show that a left quasigroup $X$ right linear over a group $(G,+)$ is a cycle set if and only if 
\begin{equation}\label{rightlinearcond}
    \phi(\phi(y)+\psi(x))+\psi(\phi(y))=\phi(\phi(x)+\psi(y))+\psi(\phi(x))
\end{equation}
holds for all $x,y\in G$. In a similar way, it follows that a left quasigroup $X$ affine over a group $(G,+)$, where $x\cdotp y=\phi(x)+\psi(y)+c$ for a constant $c\in G$, is a cycle set if and only if $\phi\psi(x)- \psi\phi(x)=\phi^2(x)$ for all $x\in G$.
The automorphism group of affine latin left quasigroups are completely described.

\begin{lemma}(\cite[Proposition $3.9$]{drapal2009group})\label{autquasi}
    Let $X=\aff{G,\phi,\psi,c}$ be a latin left quasigroup affine over an abelian group $(G,+)$. Then 
$$Aut(G,\cdotp)=\{ t_g\circ \eta\mid \eta \in C_{Aut(G,+)}(\phi,\psi),\quad \eta(c)-c=\phi(g)+\psi(g)-g\}$$
\end{lemma}

Deformations preserve in several cases the property of being affine or right linear, as we can see in the following proposition.

\begin{prop}\label{invariant}
    Let $X$ be a left quasigroup, $\alpha\in \Aut{(X,\cdot)}$ and $G$ be a group. Then:
    \begin{itemize}
        \item[1)] $X$ is irretractable (resp. latin) if and only if $X_\alpha$ is irretractable (resp. latin).
      \item[2)] $\dis(X)=\dis(X_\alpha)$. 

        \item[3)] If $G$ is abelian then $X$ is affine over $G$ if and only if $X_\alpha$ is affine over $G$. 
        \item[4)] If $X$ is a latin cycle set, then $X$ is right linear over $G$ if and only if $X_\alpha$ is right linear over $G$.
    \end{itemize}
\end{prop}
\begin{proof}
1) Let $x,y\in X$ and let $\bar{\sigma}_x$, $\bar{\delta}_x$ denote the left and right multiplication by $x$ in $X_\alpha$. Note that $\bar{\sigma}_x=\alpha \sigma_x$ and $\bar{\delta}_x=\alpha\delta_x$.

Clearly $\sigma_x=\sigma_y$ if and only if $\bar{\sigma}_x=\bar{\sigma}_y$ and $\delta_x$ is bijective if and only if $\bar{\delta}_x$ is bijective. Hence $X$ is irretractable (resp. latin) if and only if $X_\alpha$ is irretractable (resp. latin).

2) Since $\alpha\in Aut(X)$, we obtain $\bar{\sigma}_x  \bar{\sigma}^{-1}_y =\alpha\sigma_x\sigma_y^{-1}\alpha^{-1}=\sigma_{\alpha(x)}\sigma^{-1}_{\alpha(y)}$, hence $\dis(X)=\dis(X_\alpha)$.
    
    3) If $X$ is a quasigroup affine over an abelian group $G$ and $\alpha\in Aut(X,\cdotp )$, by \cref{autquasi} there exist $\eta \in Aut(G,\circ)$ and $g\in G$ such that $x\cdotp_{\alpha } y=g\circ \eta(\phi(x)\circ \psi(y)\circ c)$ for all $x,y\in X$, and therefore $x\cdotp_{\alpha } y=\eta \phi(x)\circ \eta\psi(y)\circ (g\circ \eta(c))$ for all $x,y\in X$. Hence $X_{\alpha}$ is affine over $(G,\circ )$. Since $\alpha^{-1}\in Aut(X,\cdotp_{\alpha})$ and $(X,\cdotp)$ can be obtained by the deformation of $X_\alpha$ by $\alpha^{-1}$, the converse also follows.
    
4) By 1), 2) and \cite[Theorem $5.3$]{bon2019}, we have that if $X$ is a latin cycle set right linear over a group $G$, then so is $X_{\alpha}$. The converse follows as in 3).
\end{proof}





\begin{defin}
    Let $G$ be a group, $H$ a normal subgroup of $G$ and $f$ a map from $G$ to itself. Then, $H$ is said to be \emph{invariant under $f$} if $x^{-1}y\in H$ implies that $f(x)^{-1}f(y)\in H$ for every $x,y\in G$.
\end{defin}

Let us show a description of imprimitive blocks systems and congruences of right linear left quasigroups.

\begin{prop}\label{primap}
        Let $X$ be a finite latin left quasigroup right linear over a group $(G,\circ)$. Then, the complete blocks systems of $\mathcal{G}(X)$ are in one-to-one correspondence with the subgroups of $G$ that are invariant under $\psi$.
\end{prop}

\begin{proof}
    Let $H$ be a subgroup of $G$ invariant under $\psi$ and set $B_{x,H}:=\{x\circ h\mid h\in H\}$ for all $x\in G$ and $\mathcal{B}_H:=\{B_{x,H}\}_{x\in G}$. Then, by a standard calculation one can show that $\mathcal{B}_H$ is a complete blocks system. 
    
    Conversely, let $\mathcal{B}$ be a complete blocks system and let $B\in \mathcal{B}$ such that $0\in B$. Let $y=\phi^{-1}(0)$. Then $\sigma_y=\psi$ and so $\sigma_y(0)=0$, and accordingly $\psi(B)=\sigma_y(B)=B$.
    
    We indicate by $\sim$ the equivalence relation induced by $\mathcal{B}$. At first, we show that $B$ is a subgroup of $G$. Now, let $y_1,y_2\in B$. Then, since $\phi$ is a bijection, we have that $s\circ \psi(y_1)\sim s \circ \psi(y_2)$ for all $s\in G$. If we set $s:=\psi(y_1)^-$, since $\psi\in Aut(G,\circ)$ we obtain $\psi(y_1^-\circ y_2)\sim 0$, therefore $\psi(y_1^-\circ y_2)\in B$, and hence $y_1^-\circ y_2\in B$. Therefore, it follows that $B$ is a subgroup of $G$. By a standard calculation, we obtain that $\mathcal{B}=\mathcal{B}_B$, hence the statement follows.
\end{proof}

The following result improves \cite[Lemma 4.8]{bon2019} in the finite case. 

\begin{prop}\label{congrp}
    Let $X$ be a finite latin left quasigroup right linear over a group $(G,\circ)$. Then, the congruence of $X$ are in one-to-one correspondence with normal subgroups $H$ of $G$ that are invariant under $\psi$ and $\phi$.
\end{prop}

\begin{proof}
   Let $\sim$ be a congruence of $X$. By \cref{primap} (using the same notation) there exist a $\psi$-invariant subgroup $H$ of $G$ such that $\sim$ coincides with the equivalence relation induced by $\mathcal{B}_H$. Now, if $x\sim x'$ we have that $x\cdotp 0\sim x' \cdotp 0$, hence $\phi(x)\sim \phi (x')$ and moreover $\phi^{-1}(0)\cdotp x\sim \phi^{-1}(0)\cdotp x$, therefore $H$ is invariant under $\phi$. From the finitneness of $X$ it follows that $\phi^k(x)\sim \phi^k (x')$ and $\psi^k(x)\sim \psi^k(x')$ for all $k\in \mathbb{Z}$. Now, let $x,x',y,y'\in X$ such that $x\sim x'$ and $y\sim y'$. Then, since $\sim$ is a congruence of $X$, we have $\phi^{-1}(x)\cdotp \psi^{-1}(y) \sim \phi^{-1}(x')\cdotp \psi^{-1}(y')$ and hence $x\circ y\sim x'\circ y'$, i.e. $\sim$ is a congruence of $(G,\circ)$. Therefore $H$ is a normal subgroup. 
   
   The converse is a standard calculation.
\end{proof}


\subsection{Right linear representation}

A left quasigroup $X$ has semiregular displacement group if and only if $\sim_{\dis(X)}=\sim_{1_X}$ (see \cref{sim H}). Left quasigroups with semiregular displacement groups have been studied in \cite{Nilpotent}, and in particular we have the following corollary of \cite[Lemma 3.2]{Nilpotent}.

\begin{prop}\label{risureg}
    Every finite irretractable left quasigroup with semiregular displacement group is a latin.
\end{prop}



The connection between right linear left quasigrups and cycle sets has been studied in \cite{bon2019}, in particular, we have the following reformulation of \cite[Theorem $5.3$]{bon2019}, since finite irretractable cycle sets with semiregular displacement group are latin.


\begin{cor}\label{carrightlin}\label{secondp}
    For a finite irretractable cycle set $X$ the following conditions are equivalent:
    \begin{itemize}
        \item[1)] $X$ is right linear over a group $G$ with $G\cong \dis(X)$.
        \item[2)] $\dis(X)$ acts regularly on $X$.
    \end{itemize}
\end{cor}
\begin{proof}
If $X$ is an irretractable cycle set, by a standard calculation we have that if either $1)$ or $2)$ hold, then $X$ is latin. Hence, the statement follows by \cite[Theorem 5.3]{bon2019}.
\end{proof}

\begin{rem}\label{raprightlinear}
If $(X,\cdotp)$ is an irretractable cycle set with regular displacement group, by \cref{carrightlin} we can represent it as a cycle set right linear over a suitable group $(G,\circ)$ and suitable maps $\phi$ and $\psi$ as in \cref{defrightlin}. The theory contained in \cite[Section 5]{bon2019} gives a method to find the operation $\circ$ and the maps $\phi$ and $\psi$, even if these maps are not uniquely determined. Indeed, if we choose two elements $e,f\in X$, then $X$ has a group structure given by $x\circ_{e,f} y:=\delta_e^{-1}(x)\cdotp \sigma_f^{-1}(y)$ for all $x,y\in X$ , where $f\cdotp e$ is the neutral element. In this way, $\phi_{e,f}(x):=\delta_e(x)\circ_{e,f} \sigma_f(e\cdotp f)$ and  $\psi_{e,f}(x):=\sigma_f(e\cdotp f)^{-1}\circ_{e,f} \sigma_f(x)$ for all $x\in X$. The displacement group $\dis(X)$ is isomorphic to $(X,\circ_{e,f})$, and an explicit isomorphism is given by $\rho_{e,f}:(X,\circ_{e,f})\rightarrow \dis(X)$, $x\mapsto \sigma_{\delta_e^{-1}(x)}\sigma_f^{-1}$ for all $x\in X$. The map $\psi_{e,f}$ induces an automorphism of $\dis(X)$ given by $\overline{\psi_{e,f}}(\sigma_{\delta_e^{-1}(x)}\sigma_f^{-1}):=\sigma_{\delta_e^{-1}(\psi_{e,f}(x))}\sigma_f^{-1}$ and the map $\phi_{e,f}$ induces a bijection given by $\overline{\phi_{e,f}}(\sigma_{\delta_e^{-1}(x)}\sigma_f^{-1}):=\sigma_{\delta_e^{-1}(\phi_{e,f}(x))}\sigma_f^{-1}$, for all $x\in X$. Therefore, under the correspondence $\rho_{e,f}$, the underlying group $(X,\circ_{e,f})$ can be identified with the displacement group and without loss of generality we can say that $X$ is right linear over $\dis(X)$.
\end{rem}



By the main theorem of \cite{cedo2020primitive} and \cref{quasipr}, we have that  the orbits of a minimal normal subgroup of $\mathcal{G}(X)$ forms a non-trivial complete blocks system. 
\cref{primap} and \cref{congrp} allow to describe all the complete blocks systems and all the congruences of irretractable cycle sets with regular displacement group in terms of normal subgroups and ideals contained in $\dis(X)$.

\begin{cor}\label{imprblo}\label{corrgr}
    Let $X$ be an irretractable cycle set with regular displacement group and let $e,f$ be elements of $X$. Then:
    \begin{itemize}
        \item[1)] the complete blocks systems of $X$ are in one-to-one correspondence with the subgroups of $\dis(X)$ that are invariant by $\overline{\psi_{e,f}}$;
        \item[2)] the congruence of $X$ are in one-to-one correspondence with normal subgroups $H$ of $\dis(X)$ that are invariant under $\overline{\psi_{e,f}}$ and $\overline{\phi_{e,f}}$.
    \end{itemize}   
\end{cor}
\begin{proof}
    Both $1)$ and $2)$ follow by \cref{secondp}, \cref{primap}, \cref{congrp}, and \cref{raprightlinear}. 
\end{proof}

    





The congruences of $X$, when it is an irretractable cycle set with regular displacement group, can be described in terms of the left brace $\mathcal{G}(X)$.

\begin{theor}\label{idcont}
     Let $X$ be a finite irretractable cycle set with regular displacement group $\dis(X)$. Then, the congruences of $X$ are in one-to-one correspondence with ideals of $\mathcal{G}(X)$ contained in $\dis(X)$.
\end{theor}

\begin{proof}
    By \cref{secondp} and \cref{raprightlinear}, $X$ can be identified with a cycle set right linear over a group $(G,\circ)$ and given by $x\cdotp y:=\phi(x)\circ \psi(y)$ for all $x,y\in G$, for a suitable maps $\phi$ and $\psi$. By \cref{congrp} (using the same notation) the congruences of $X$ are in one-to-one correspondence with normal subgroups $H$ of $G$ that are invariant under $\phi$ and $\psi$ and the correspondence is given by $H\mapsto \mathcal{B}_H$. So the elements of $\mathcal{B}_H $ correspond to the orbits respect to the action of the subgroup $T_H$ of $\dis(X)$ given by the translations by the elements of $H$. Therefore, since $X$ is irretractable, by \cite[Lemma 2.1]{bachiller2015family} and \cite[Theorem 1]{rump2007braces} $T_H$ is an ideal of $\mathcal{G}(X)$. Therefore, every congruence comes from a suitable ideal of $\mathcal{G}(X)$ contained in $\dis(X)$. Since $\dis(X)$ acts by left translation with respect to the operation $\circ$, different ideals contained in $\dis(X)$ give rise to different congruences, hence the statement follows.
\end{proof}


The following result describes the left brace structure of the permutation group $\mathcal{G}(X)$ of an irretractable cycle set with regular displacement group of prime-power size.

\begin{prop}\label{permx}
    Let $X$ be a finite irretractable cycle set with regular displacement group and $e,f$ be elements of $X$. Then, the permutation group $\mathcal{G}(X)$ is isomorphic to $\dis(X)\rtimes <\psi_{e,f}>$. 
\end{prop}
\begin{proof}
    It follows by \cref{secondp} and  \cref{raprightlinear}, toegether with a standard calculation.
\end{proof}


\begin{cor}\label{permx2}
    Let $p$ be a prime. Let $X$ be an irretractable cycle set with regular displacement group of size $p^k$. Then, $\mathcal{G}(X)$, as a left brace, is isomorphic to the semidirect product of $B_p$ and $C_n$, where $B_p$ is the $p$-Sylow subgroup of $(\mathcal{G}(X),+)$ and $C_n$ is a trivial cyclic left brace and $gcd(n,p)=1$. 
\end{cor}

\begin{proof}
    It follows by \cref{secondp}, \cref{permx} and \cref{struttbr}.
\end{proof}

As a final remark of this section, note that, since an abelian transitive subgroup is regular, an irretractabe cycle set $X$ with abelian displacement group $\dis(X)$ is latin if and only if $\dis(X)$ acts transitively on $X$. Therefore, our theory applies to all cycle sets with regular and abelian displacement group.

\section{Simple cycle sets}

\subsection{Simple left quasigroups}

Let us turn our attention to simple left quasigroups.
    \begin{defin}
    Let $X$ be an indecomposable left quasigroup. Then, $X$ is \emph{simple} if the only congruences of $X$ are $\sim_{0_X}$ and $\sim_{1_X}$. 
\end{defin}

\begin{cor}\label{carattsimple2}
    Let $X$ be an indecomposable and irretractable left quasigroup. Then, $X$ is simple if and only if $\dis(X)$ is the smallest nonzero admissible subgroup. Moreover, $\dis(X)$ acts transitively on $X$.
\end{cor}

\begin{proof}
    Suppose that $X$ is simple. Clearly, $\dis(X)$ is a nonzero admissible subgroup and by the theory after \cref{admissiblesub} it acts transitively on $X$. For the same reason, every admissible subgroup $H$ must acts transitively on $X$ and, by the definition of admissible subgroup, $\dis(X)$ is contained in $H$, hence $\dis(X)$ is the smallest nonzero admissible subgroup of $\mathcal{G}(X)$.\\
    Conversely, suppose that $\dis(X)$ is the smallest nonzero admissible subgroup that acts transitively on $X$. Then, this implies the transitivity of $H$ on $X$, hence simplicity of $X$ follows by \cref{fattorizzleft}.
\end{proof}

From now on, a cycle set will be called simple if its underlying left quasigroup is simple. As it is well-known that if $X$ is an indecomposable cycle set then the ideal $\mathcal{G}(X)^{2}$ of $\mathcal{G}(X)$ coincides with $\dis(X)$ (see \cite[Lemma 11]{castelli2023indecomposable}), we can use our results to recover a characterization of simple cycle sets. Since the unique simple decomposable cycle set has size $2$, we decide to add the condition $|X|>2$ and isolate this easy case.

\begin{cor}\label{carattsempl}
    Let $X$ be a finite cycle set and suppose that $|X|>2$. Then, $X$ is simple if and only if $X$ is the indecomposable trivial cycle set of prime size $p$ or $X$ is irretractable and $\mathcal{G}(X)^2$ is a minimal ideal of $\mathcal{G}(X)$ that acts transitively on $X$.
\end{cor}

\begin{proof}
    If $X$ is simple and is not the trivial cycle set of prime size, by \cite[Proposition 4.1]{cedo2021constructing} it is irretractable, hence this implication follows by \cref{idadm}, \cref{carattsimple2} and the equality $\mathcal{G}(X)^2=\dis(X)$. Conversely, if $X$ is the trivial indecomposable of prime size, then it is simple by \cite[Lemma 1]{cacsp2018}. Now, suppose that $X$ is irretractable and $\mathcal{G}(X)^2$ is a minimal ideal of $\mathcal{G}(X)$ that acts transitively on $X$. Then $X$ is indecomposable and hence we have $\mathcal{G}(X)^2=\dis(X)$. Therefore, the thesis follows by \cref{idadm} and \cref{carattsimple2}.
\end{proof}
\subsection{Simple cycle sets with nilpotent displacement group}

In this section we provide some limitations on the cardinality of simple cycle sets having nilpotent permutation group. In particular, we show that simple cycle sets with cyclic displacement group admit a very simple description.


\begin{lemma}\label{ris1}
    Let $X$ be an indecomposable cycle set and assume that $\dis(X)$ has a non-trivial normal Hall subgroup $H$. Then, $H$ is an ideal of $\mathcal{G}(X)$.
\end{lemma}

\begin{proof}
    If $H$ is a non-trivial normal Hall subgroup of $\dis(X)$, then by \cite[Proposition 4.2.1]{bachiller2016study} it is an ideal of $\dis(X)$. Moreover, $H$ is a characteristic subgroup of $(\dis(X),+)$ and $(\dis(X),\circ)$, therefore it is an ideal of $\mathcal{G}(X)$. 
\end{proof}

\begin{cor}\label{ris2}
    Let $X$ be a simple cycle set with nilpotent displacement group. Then, $|X|=p^k$ for a prime number $p$ and a natural number $k$.
\end{cor}

\begin{proof}
    If $p,q$ are distinct prime numbers dividing $|X|$, then the $q$-Sylow of $(\dis(X),\circ)$ is a non-trivial normal Hall subgroup of $(\dis(X),\circ)$. By \cref{ris1} and \cref{epiid} $X$ has a non-trivial congruence, a contradiction.
\end{proof}

Now, we can describe all the simple cycle sets with cyclic displacement group.


\begin{theor}\label{cardpr}
    Let $X$ be a finite simple cycle set. Then, the following statements are equivalent:
    \begin{itemize}
        \item[1)] $\dis(X)$ is cyclic;
         \item[2)] $|\dis(X)|=1$;
         \item[3)] $|X|=2$ or $|X|$ is an indecomposable cycle set of prime size.
    \end{itemize}
\end{theor}

\begin{proof}
1) $\Rightarrow$ 3)   Suppose that $X$ is a simple cycle set with cyclic displacement group and such that $|X|$ is not a prime number. Then, $X$ must be irretractable and hence $\dis(X)$ acts transitively on $X$, therefore by \cref{ris2} we have $|X|=|\dis(X)|=p^k$ for some prime number $p$. Moreover, by an inspection of the small cycle sets, $|X|\neq 4$. Now, let $m$ be a divisor of $|X|$ such that $|X|/m=p$ for a prime number $p$. By the results containted in \cite{rump2007classification} and \cite[Proposition 5.4]{bachiller2016solutions}, $\dis(X)$ is a left brace with cyclic additive and multiplicative group. Moreover, the set $H:=p \dis(X)$ is an ideal of $\dis(X)$ of size $m$, and is a characteristic subgroup of $(\dis(X),+)$ and $(\dis(X),\circ)$, hence it is an ideal of $\mathcal{G}(X)$. Therefore, $p\dis(X)$ induces a non-trivial congruence on $X$, a contradiction. Then, $|X|$ must be a prime number, and by \cite[Proposition 4.1]{cedo2021constructing} $|X|=2$ or $X$ is indecomposable. If the first case, clearly  $\dis(X)=\{id_X\}$ by a standard calculation. In the second case the conclusion is the same by \cite[Theorem 2.13]{etingof1998set}. 

3) $\Rightarrow$ 2) If $X$ is a simple cycle set of prime size, we have that $\dis(X)$ is the trivial group by \cite[Theorem 2.13]{etingof1998set}

2) $\Rightarrow$ 1) Clear.
\end{proof}

\subsection{Simple cycle sets with regular displacement group}

In this subsection, we provide a description of simple cycle sets with nilpotent regular permutation group. 

\begin{cor}\label{corright}
    Let $X$ be a finite simple cycle set with regular displacement group. Then, $X$ is right linear over $\dis(X)$.
\end{cor}

\begin{proof}
The simplicity of $X$ implies that $X$ is irretractable or $|\dis(X)|=1$. Since $\dis(X)$ acts regularly, $X$ must be irretractable, hence the statement follows by \cref{secondp}.
\end{proof}

From now on, in this subsection we indicate by $(p,G,\phi,\psi)$ a $4$-uple such that: 
\begin{itemize}
 \item $p$ is a prime number
        \item $(G,\circ)$ is a group of size $p^k$;
        \item $(\phi,\psi)\in Sym(G)\times Aut(G,+)$ such that \cref{rightlinearcond} holds for all $x,y\in G$.
\end{itemize}

\noindent By \cref{corrgr}, \cref{ris2} and \cref{corright}, we are able to construct all the simple cycle sets with nilpotent regular permutation group. We leave the proof to the reader since it follows by the cited results.

\begin{theor}\label{descrpr}
    Let $(p,G,\phi,\psi)$ be such that every normal subgroup $H$ of $G$ invariant under $\psi$ is not invariant under $\phi$. Then, the pair $(X,\cdotp)$ given by $X:=G$ and 
    $$x\cdotp y:=\phi(x)\circ\psi(y) $$
    for all $x,y\in G$ is a non-trivial simple cycle set with $\dis(X)\cong G$.\\ 
    Conversely, every simple cycle set $X$ having nilpotent regular displacement group can be constructed in this way.
\end{theor}


\noindent If $X$ is a simple cycle set with nilpotent regular displacement group, by \cref{defpri} it can be constructed by a suitable cycle set having a $p$-group permutation group.

\begin{cor}
    Let $Y$ be a simple cycle set having nilpotent regular displacement group. Then, it is a deformation of a right linear cycle set $X$ by a suitable $\alpha\in Aut(X)$ such that $\mathcal{G}(X)$ is a $p$-group for some prime number $p$.
\end{cor}

\begin{proof}
    It follows by \cref{ris2}, \cref{corright}, \cref{permx2}, \cref{defpri}, and  \cref{invariant}.
\end{proof}

By the previous corollary, if $Y$ is a simple cycle set  with nilpotent regular displacement group, then it is a deformation of some cycle set $X$ by an automorphism $\alpha$. Unfortunately, we do not know if such an $X$ is simple. Moreover, we do not know if the deformation of a simple cycle set is again simple, but we are able to show this under some additional hypothesis. 

\begin{prop}\label{descrsimple}
    Let $X$ be a simple cycle set having a $p$-group permutation group $\mathcal{G}(X)$. Let $\alpha$ be an automorphism of $X$ having order coprime with $p$ and fixing an element $x\in X$. Then, the deformation $X_{\alpha}$ of $X$ is simple.
\end{prop}

\begin{proof}
    Suppose that $X_{\alpha}$ is not a simple cycle set. Then, by \cite[Lemma 4.5]{dietzel2023indecomposable} the left brace $\mathcal{G}(X_{\alpha})$  is isomorphic to $ \mathcal{G}(X)\rtimes <\alpha>$, and by \cref{carattsempl} there exist a nonzero ideal $J$ of $\mathcal{G}(X_{\alpha})$ that does not act transitively on $X_{\alpha}$. Then, the ideal $J\cap \dis(X_{\alpha})$, which is equal to $J\cap \mathcal{G}(X_{\alpha})^2$, contains $J*\mathcal{G}(X)$, therefore if $J\cap \dis(X_{\alpha})=0$ we have $J*\mathcal{G}(X_{\alpha})=0$ and hence $J\subseteq Soc(\mathcal{G}(X_{\alpha}))=0$, but this is not possible because $J$ is nonzero. Hence $J\cap \dis(X_{\alpha})$ is a nonzero ideal of $\mathcal{G}(X_{\alpha})$ contained in $\dis(X_{\alpha})$. Since by \cref{invariant} $\dis(X_{\alpha})=\dis(X)$, we have that $J\cap \dis(X)$ is a nonzero ideal of $\mathcal{G}(X)$ that acts not transitively on $X$, hence by \cref{carattsempl} $X$ is not simple, a contradiction.
\end{proof}

\begin{cor}\label{cossimple}
    Let $(p,G,\phi,\psi)$ be such that $\psi$ has order $p^s$ for some natural number $s$, and every normal subgroup $H$ of $G$ invariant under $\psi$ is not invariant under $\phi$. Moreover, let $\alpha \in Aut(G,\cdotp)$, where $x\cdotp y:=\phi(x)\circ\psi(y)$ for all $x,y\in G$, such that $\alpha$ has coprime order with $p$ and $\alpha(z)=z$ for some $z\in G$. Then, the pair $(X,\bullet)$ given by $X:=G$ and 
    $$x\bullet y:=\alpha(\phi(x)\circ\psi(y)) $$
    for all $x,y\in G$ is a simple cycle set. 
\end{cor}

\begin{proof}
    It follows by \cref{permx2}, \cref{descrpr} and \cref{descrsimple}.
\end{proof}

\section{Affine cycle sets}

In this section we specialize the results of the previous sections to latin affine cycle sets. In this case, the cycle set can be identified with an abelian group $(X,+)$ and the $\cdotp$ operation is defined by
\begin{equation}\label{eqaffini}
    x\cdotp y:=\phi(x)+\psi(y)+c
\end{equation}
for all $x,y\in X$, where $\phi,\psi\in Aut(X,+)$ and $c\in X$. Recall that we denote such a cycle set by $\aff{X,\phi,\psi,c}$

\smallskip

\noindent The following result, that is a reformulation of \cite[Theorem $4.18$]{bon2019}, allows to detect the class of affine cycle sets by group-theoretic properties.

\begin{cor}\label{secondp2}
    Let $X$ be a finite irretractable cycle set with abelian displacement group $\dis(X)$. Then, $X$ is affine if and only if $\dis(X)$ is a normal subgroup of $\mathcal{TG}(X)$ that acts regularly on $X$.
\end{cor}

\begin{proof}
    It follows by \cref{risureg} and \cite[Theorem $4.18$]{bon2019}.
\end{proof}


As highlited in \cite{bon2019}, there is a relation between affine cycle sets and the first Weyl algebra over a field.

\begin{rem}\label{collweyl}
Let us define the {\it first Weyl algebra} over a field $K$ as $A_1(K)=K[a,b]/\langle ab-ba-1\rangle$.
\begin{enumerate}
    \item    Let $n$ be a natural number, $K$ be a field and $(X,\cdotp)=\aff{K^n,\phi,\psi,c}$ be a latin affine cycle set with underlying abelian group $(K^n,+)$ and with $\psi,\phi$ automorphisms of the canonical vector space $K^n$. As showed in \cite[Section $4$]{bon2019}, in this case \cref{eqcycleset} is equivalent to the relation $\psi\phi^{-1}-\phi^{-1}\psi=id_{X}$, that is also the defining relation of the first Weyl Algebra. Thus we have an $n$-dimensional representation $\rho$ of the first Weyl algebra $A_1(K)$ defined by $\rho:A_1(K)\longrightarrow \End{(K^n,+)}$ with $\rho(a)=\psi$ and $\rho(b)=\phi^{-1}$.
    
 \item   Conversely, if $\rho$ is an $n$-dimensional representation, defined by $\rho:A_1(K)\rightarrow \End{(K^{n},+)}$ with $\rho(a),\rho(b)\in \Aut{(K^{n},+)}$, then $\aff{K^n,\rho(b)^{-1},\rho(a),c}$ gives rise to an affine cycle set for every $c\in K^{n}$. 
\end{enumerate}

Note that, if $(K,+)$ is a field having a prime number $p$ of elements, the dimension of the representation must be a multiple of $p$ \cite[Proposition 4.7]{bon2019}. 
\end{rem}

\subsection{Simple affine cycle sets}
%


The properties of being affine and simple force the displacement group to be elementary abelian.

\begin{prop}\label{elab}
    Let $X$ be a simple affine latin cycle set. Then, $\dis(X)$ is an elementary abelian $p$-group of size $p^k$ for some prime number $p$ dividing $k$.
\end{prop}

\begin{proof}
  Since $X$ is affine then $\dis(X)$ is abelian. By \cref{ris2} the underlying abelian group is a p-group. Since in a finite abelian group $(G,+)$ the elements of the form $pg$, with $g\in G$, form a characteristic subgroup, the statement follows by \cref{corrgr} and \cref{raprightlinear}.
\end{proof} 

Let $p$ be a prime number. As a main consequence of \cref{collweyl} and \cref{elab}, we obtain that every finite simple affine cycle set corresponds to a finite-dimensional irreducible representations $\rho$ of $A_1(\mathbb{Z}/p\mathbb{Z})$ together with a constant $c\in (\mathbb{Z}/p\mathbb{Z})^n$, where $n:=dim(\rho)$. 

\begin{theor}\label{irrid}
    Let $p$ be a prime number and $\rho$ be an irreducible representation with dimension $n$ of $A_1(\mathbb{Z}/p\mathbb{Z})$. Suppose that $\rho(a)$ and $\rho(b)$ are invertible and set $\phi:=\rho(b)^{-1}$ and $\psi:=\rho(a)$. Then $\aff{(\mathbb{Z}/p\mathbb{Z})^n,\phi,\psi,c}$is a simple latin affine cycle set for every $c\in \mathbb{Z}/p\mathbb{Z}$. Moreover, $\mathcal{G}(X)$ is a $p$-group if and only if $\psi$ has prime-power order.\\
    Conversely, every affine simple latin cycle set can be constructed in this way.
\end{theor}

\begin{proof}
    Let $\rho$ be as in the statement. Then, the binary operation $\cdotp$ on  $(\mathbb{Z}/p\mathbb{Z})^n$ given by $x\cdotp y:=\phi(x)+\psi(y)+c$ makes $ (\mathbb{Z}/p\mathbb{Z})^n$ into an affine latin cycle set by Remark \ref{collweyl} . Since $\rho$ is irreducible and in $(\mathbb{Z}/p\mathbb{Z})^n$ subgroups concide with subspaces, by \cref{corrgr} we cannot have non-trivial congruences, hence the constructed cycle set is simple. By \cref{permx} we have that $\mathcal{G}(X)$ is a $p$-group if and only if $\psi$ has prime-power order.\\
    Conversely, suppose that $X$ is an affine simple latin cycle set. Then, by \cite[Section 4]{bon2019} and \cref{elab} there exist a natural number $n$, a prime number $p$ dividing $n$, an element $c\in (\mathbb{Z}/p\mathbb{Z})^n$, and two automorphism $\phi$ and $\psi$ of $(\mathbb{Z}/p\mathbb{Z})^n$, such that $X:=(\mathbb{Z}/p\mathbb{Z})^n$ and $x\cdotp y:=\phi(x)+\psi(y)+c$ for all $x,y\in (\mathbb{Z}/p\mathbb{Z})^n$. Then, the assignments $a\mapsto \psi$ and $b\mapsto \phi^{-1}$ gives rise to a representation $\rho$ of $A_1(\mathbb{Z}/p\mathbb{Z})$ (see also Remark \ref{collweyl}). By simplicity of $X$, the irreducibility of $\rho$ follows. 
\end{proof}

\begin{rem}
    We highlight that by \cref{secondp2}, the simple cycle sets $X$ with displacement group abelian and normal in the total permutation group $\mathcal{TG}(X)$ are all and only the ones provided in the previous theorem.
\end{rem}

Now, we indicate by $(p,\rho,c,\eta,g)$ a quintuple such that:
\begin{itemize}
    \item $p$ is a prime number;
    \item $\rho$ is a representation of the Weyl Algebra $A_1(\mathbb{Z}/p\mathbb{Z})$ with $\rho(a):=\psi$ and $\rho(b):=\phi$ invertible matrices;
    \item $c\in (\mathbb{Z}/p\mathbb{Z})^n $, where $n:=dim(\rho)$;
    \item $(\eta,g) \in C_{Aut(G,+)}(A,B)\times (\mathbb{Z}/p\mathbb{Z})^n$ is such that $\eta(c)-c=Ag+Bg-g$.
\end{itemize}
Below, we show an affine-version of \cref{cossimple}. The proof is left to the reader, since it follows by the previous results and is standard.

\begin{cor}
   Let  $(p,\rho,c,\eta,g)$ be a quintuple such that $\rho$ is an irreducible representation, $\psi$ as order $p^s$ for some natural number $s$, and $\eta$ has order coprime with $p$.  Then, the pair $((\mathbb{Z}/p\mathbb{Z})^n,\cdotp)$ with  
 $$x\cdotp y:=g+\eta(\rho(b)^{-1}x+\rho(a)y+c)$$
 for all $x,y\in (\mathbb{Z}/p\mathbb{Z})^n$ is a simple cycle set having $\dis(X)$ abelian, regular and normal in $\mathcal{TG}(X)$.
\end{cor}



We close this subsection showing that even if, in full generality, an affine cycle set is not simple (for example, take an affine cycle set having size $p^\alpha q^\beta$, for distinct prime numbers $p,q$), it always provides simple non-trivial cycle sets.
\begin{prop}
    Let $X$ be a finite cycle set affine over a group $G$ and given by $x\cdotp y:=\phi(x)+\psi(y)+c$. Let $H$ be a subgroup of $G$ that is minimal among the subgroups invariant under the actions of $\phi$ and $\psi$. Let $h\in H$ and define $Y$ as the cycle set given by $Y:=H$ and $x\cdotp y:=\phi(x)+\psi(y)+h$ for all $x,y\in Y$. Then, $Y$ is a simple non-trivial cycle set affine over $H$.
\end{prop}
 \begin{proof}
By a standard calculation, one can show that $Y$ is a non-trivial cycle set affine over $H$. If $Y$ is not simple, there exist a subgroup $K$ of $H$ that is invariant under the actions of $\phi$ and $\psi$. This implies that $H$ is not minimal among the subgroups of $G$ invariant under the actions of $\phi$ and $\psi$, a contradiction. Hence $Y$ is simple.
\end{proof}

\subsection{Simple affine cycle sets of minimal size}

By \cref{irrid}, several examples of simple cycle sets can be obtained by irreducible representations of the first Weyl Algebra $A_1(\mathbb{Z}/p\mathbb{Z})$. Even if these representations are not classified over arbitrary fields (to our knowledge), several examples can be recovered from the irreducible representations of the first Weyl Algebras on algebrically closed fields, that are completely classified. Indeed, if $K$ is an algebrically closed field of characteristic $p>0$, all the finite dimensional irreducible representations of $A_1(K):=K[a,b]/(ab-ba-1)$ are given by the $p\times p$ matrices 
\begin{equation}\label{canonrapp}
    M_a:=\begin{bmatrix}
    0 & 0 & 0 & ... & \mu \\
    1 & 0 & 0 & ... & 0 \\
    0 & 1 & 0 & ... & 0 \\
    \vdots & \vdots  & \vdots & \vdots & \vdots  \\
    0 & 0 & 0 &  1 & 0
\end{bmatrix}
\qquad and \qquad
M_b:=\begin{bmatrix}
    \lambda & p-1 & 0 & ... & 0 \\
    0 & \lambda & p-2 & ... & 0 \\
    0 & 0 & \lambda & ... & 0 \\
    \vdots & \vdots  & \vdots & \vdots & 1  \\
    0 & 0 & 0 &  0 & \lambda
\end{bmatrix}
\end{equation}
where $\mu,\lambda\in K$. The argument for the proof of the following Proposition was suggested in \cite{overflow}.

\begin{prop}\label{matovp2}
    Let $p$ be a prime number, $V:=(\mathbb{Z}/p\mathbb{Z})^p$ and  $\rho:A_1(\mathbb{Z}/p\mathbb{Z})\rightarrow End(V)$ be an irreducible representation of $A_1(\mathbb{Z}/p\mathbb{Z})$ of dimension $p$. Then, there exists a basis $\{e_1,...,e_p\}$ of $V$ such that $\rho(a)=M_a$ and $\rho(b)=M_b$ for suitable elements $\mu,\lambda\in \mathbb{Z}/p\mathbb{Z}$. 
\end{prop}

\begin{proof}
    Let $M_1:=\rho(a)$, $M_2:=\rho(b)$ with respect some basis $\mathcal{B}=\{e_1,\ldots,e_p\}$ of $V$, and $K$ be the algebraic closure of $\mathbb{Z}/p\mathbb{Z}$. Since $dim(V)=p$, $M_1$ and $M_2$ provide a $p$-dimensional representation $\bar{\rho}$ of $A_1(K)$, that must be irreducible. Moreover, we have that $\mathcal{B}$ can be regarded as a basis of $K^p$, and we have $\bar{\rho}(a)=M_1$ and $\bar{\rho}(b)=M_2$ with respect $\mathcal{B}$. 
    By the previous remark, there exist a basis $\mathcal{B}':=\{e'_1,...,e'_p\}$ of $K^p$ such that $\bar{\rho}(a)=M_a$ and $\bar{\rho}(b)=M_b$ for suitable elements $\mu,\lambda\in K$. Hence, there exist a matrix $C\in GL_p(K)$ providing the change of basis from $\mathcal{B}'$ to $\mathcal{B}$ such that $M_a=C^{-1}M_1 C$ and $M_b=C^{-1}M_2 C$. To show the statement, is sufficient proving that $C\in GL_p(\mathbb{Z}/p\mathbb{Z})$.\\
    The characteristic polynomial does not change by conjugacy, therefore we have that the characteristic polynomial of $M_2$, which is equal to the one of $M_b$, is $(x-\lambda)^p=x^p-\lambda^p$, hence $\lambda^p\in \mathbb{Z}/p\mathbb{Z}$ and this implies that $\lambda$ is an element of $ \mathbb{Z}/p\mathbb{Z}$. 
    Since the dimension of the $\lambda$-eigenspace of $M_2$ as endomorphism of $(\mathbb{Z}/p\mathbb{Z})^p$ (which is equal to the dimension of the $\lambda$-eigenspace of $M_b$ as endomorphism of $K^p$) is $1$, the vector $e'_1$, up to multiply by a suitable scalar, is an eigenvector of $M_2$. Therefore, we can assume that $e'_1=e_1$ (up to consider matrices $M_1'$ and $M_2'$ similar to $M_1$ and $M_2$ in $GL_p(\mathbb{Z}/p\mathbb{Z})$). 
Now, let $f_i\in \mathbb{Z}/p\mathbb{Z}^p$ having $0$ in the $j$-th position, with $j\neq i$, and $1$ in the $i$-th position. Therefore, the first column of $C$ is $f_1$. Now, we have that $M_{a}^{j-1} f_1=C^{-1} M_{1}^{j-1} C f_1=C^{-1} M_1^{j-1} f_1$ and hence $C f_j= M_{1}^{j-1} f_1 $ for every $j\in \{1,...,p\}$, and in a similar way $\mu C f_1= M_{1}^{p} f_1 $. From these equalities, we obtain $\mu\in \mathbb{Z}/p\mathbb{Z}$ and $C\in GL_p(\mathbb{Z}/p\mathbb{Z})$.

\end{proof}

By \cref{irrid} and \cref{matovp2}, we are able to construct concretely \emph{all} the affine simple cycle sets of size $p^p$, for an arbitrary prime number $p$. The next goal is to distinguish the isomorphism classes.

\noindent Now, we define the algebra of the $\mu$-circulant matrices over a field $K$ in analogy with the classical circulant matrices. Recall that a square matrix is \emph{circulant} if it is constant on all broken diagonals, and denote by $Circ(c_1,\ldots, c_n)$ the $n \times n$ circulant matrix with first row equal to $(c_1,\ldots, c_n)$. Let $\mu\in K$ we define the \emph{$\mu$-circulant} matrix $\mu-Circ(c_1,\ldots, c_n)$ to be the matrix obtained by $Circ(c_1,\ldots,c_n)$ multiplying all the upper-triangular entries of $Circ(c_1,\ldots, c_n)$ by the element $\mu$, as
\begin{equation}
\mu-Circ(c_1,\ldots, c_n)=\begin{bmatrix}
c_1 & \mu c_2 & \ldots & \ldots& \mu c_{n-1} & \mu c_n \\
c_n & c_1 & \mu c_2& \ldots &\ldots & \mu c_{n-1}\\
c_{n-1} & c_n & c_1 & \mu c_2 & \ldots& \ldots\\
\vdots & \vdots &\vdots & \vdots & \vdots& \vdots\\
c_3 & \ldots & \ldots &c_n &c_1 & \mu c_2\\
c_2 & c_3 & \ldots &c_{n-1} & c_n & c_1
\end{bmatrix}\, .
\end{equation}

The algebra of the circulant matrices coincides with the centralizer of the permutation matrix representing the $n$-cycle. We will show that the $\mu$-circulant matrices are the centralizer of the following matrix:
\begin{equation}\label{A_a def}
A_\mu=\begin{bmatrix}
0 & 0 & \ldots & 0 & \mu\\
1 & 0 & \ldots & 0 & 0\\
0 & 1 & \ldots & 0 & 0\\
\vdots & \vdots &\vdots & \vdots & \vdots\\
0 & 0 & \ldots & 1 & 0\\
\end{bmatrix}.
\end{equation}


\begin{lemma}\label{centralizer of A_a}
The set of the $n\times n$ $\mu$-circulant matrices is the subalgebra of the matrices which commutes with $A_\mu$ (as defined in \eqref{A_a def}).
\end{lemma}

\begin{proof}
Let $C=(c_{i,j})$ be a $n\times n$ matrix. Then
\begin{displaymath}
(C A_\mu)_{i,j}=
    \begin{cases}
c_{i,j+1}, \quad \text{if } j\neq n,\\
\mu c_{i,1},\quad \text{if } j=n,
\end{cases}, \quad (A_\mu C)_{i,j}=
    \begin{cases}
c_{i-1,j}, \quad \text{if } i\neq 1,\\
\mu c_{n,j},\quad \text{if } i=1.
\end{cases}
\end{displaymath}
Therefore $C$ commutes with $A_\mu$ if and only if 
\begin{displaymath}
\begin{cases}
c_{i,j+1}=c_{i-1,j}, \, \text{ for every } i\neq 1, \, j\neq n,\\
\mu c_{i,1}=c_{i-1,n}, \, \text{ for every } i\neq 1,\\
c_{1,j+1}=\mu c_{n,j}, \, \text{ for every } j\neq n,\\
\mu c_{1,1}= \mu c_{n,n}.
\end{cases}
\end{displaymath}
Hence, the centralizer of $A_\mu$ coincides with the set of the $\mu$-circulant matrices.
\end{proof}

Let $\D = (d_{i,j})$ be the $n \times n$ matrix defined by 

$$d_{i,j}= \begin{cases} n-i, \, \text{ if } j=i+1\pmod{n},\\
0,\, \text{otherwise}.\end{cases}
$$ 
e.g. if $n=3$ we have
$$\D=\begin{bmatrix}  0& 2 & 0 \\ 0& 0 & 1 \\ 0& 0 & 0\end{bmatrix}$$
\begin{lemma}\label{commutator with D}
Let $C$ be a $p\times p$ $\mu$-circulant matrix. Then $[C,\Delta]=0$ if and only if $C=s I$ for some $s\in K$.
\end{lemma}

\begin{proof}
Let $C=\mu$-$Circ(c_1,\ldots, c_p)$. It is straightforward to compute that the first column of $\Delta\cdotp C$ is $((p-1)c_{p}, (p-2) c_{p-1},\ldots,   c_2,0)$ and the first column $C\cdotp  \Delta$ is $(0,\ldots,0)$. Therefore if $[C,\Delta]=0$ holds then $c_2=c_3=\ldots=c_{p}=0$ and so $C=c_1 I$.
\end{proof}

To identify the isomorphism classes of affine cycle sets of size $p^p$ we are going to use the following criterion. The criterion is state for quasigroups and so it covers the case of finite latin cycle sets.
\begin{theor}[Dr\'apal {\cite[Thm. 3.2]{drapal2009group}}]\label{Th:IsoThm}
Let $Q=\aff{G,\phi,\psi,c}$ and $Q'=\aff{G,\phi',\psi',c'}$ be affine quasigroups. Then $Q$ is isomorphic to $Q'$ if and only if there are $\alpha\in\Aut{(G,+)}$  and $u\in\mathrm{Im}(1-\phi-\psi)$ such that $\phi'=  \alpha \phi\alpha^{-1}$, $\psi'= \alpha \psi\alpha^{-1}$ and $c'=\alpha(c+u)$.
\end{theor}

Now we are ready for the main result.

\begin{theor}\label{simple Rumples}
Let $p$ be a prime and $X$ be an affine simple latin cycle set of size $p^p$. Then $X$ is isomorphic to one of the following latin cycle sets:
\begin{eqnarray*}
\aff{\mathbb{Z}_p^p,(\lambda I+\D)^{-1},\A_\mu,(0,\ldots,0)},\quad \quad \aff{\mathbb{Z}_p^p,(\lambda I+\D)^{-1},\A_1,(1,0,\ldots,0)}.
\end{eqnarray*}
for $\mu,\lambda=1,\ldots,p-1$. In particular, there are $p(p-1)$ affine simple latin cycle sets of size $p^p$.
\end{theor}

\begin{proof}
According to \cref{irrid} and \cref{matovp2}, we have that $X$, which we indicate by $X(\lambda,\mu,c)$ for its dependence on the three parameters, is given by $X(\lambda,\mu,c)=\aff{\Z_p^p,(\lambda I+\Delta)^{-1},\A_\mu ,c}$. If $X(\lambda,\mu,c)$ and $X(\lambda',\mu',c')$ are isomorphic then the matrices $A_\mu$ and $A_{\mu'}$ (resp. $\mu I+\Delta$ and $\mu' I+\Delta$) are conjugate according to \cref{Th:IsoThm}. Therefore such matrices have the same eigenvalues, i.e. $\lambda=\lambda'$ (resp. $\mu=\mu'$). By this fact and \cref{Th:IsoThm}, $X(\lambda,\mu,c)$ and $X(\lambda',\mu',c')$ are isomorphic if and only if and only if $\lambda=\lambda'$, $\mu=\mu'$ and there exists $\alpha\in GL_p(\mathbb{Z}/p\mathbb{Z})$ and $u\in \mathrm{Im}(1-(\lambda I+\Delta)^{-1}-\A_\mu)$ such that $\mu I+\D=\alpha (\mu I+\D)\alpha^{-1}$, $\A_\mu=\alpha \A_\mu\alpha^{-1}$ and $c'=\alpha(c+u)$. By virtue of \cref{commutator with D} then $\alpha=s I$ with $s\neq 0$. Therefore
$$c'-s c=s u\in \mathrm{Im}(1-(\lambda I+\Delta)^{-1}-\A_\mu),$$
i.e. $c$ and $c'$ are generate the same subspace modulo $\mathrm{Im}(1-(\lambda I+\Delta)^{-1}-\A_\mu)$. Thus, we need to understand the dimension of $\mathrm{Im}(1-(\lambda I+\Delta)^{-1}-\A_\mu)$ in order to account how many different isomorphism classes we have. 
The dimension of $\mathrm{Im}(1-(\lambda I+\Delta)^{-1}-\A_\mu)$ equals the dimension of
$$\mathrm{Im}(\lambda I +\Delta)(1-(\lambda I+\Delta)^{-1}-\A_\mu)=\mathrm{Im}((\lambda I+\Delta)(1-\A_\mu) -1).$$
Now, we have that $(\lambda I+\Delta)(1-\A_\mu) -1$ can be written as
\begin{equation*}
  \begin{bmatrix}
\lambda & p-1& 0 &0&0 & \ldots & -\lambda\mu \\
-\lambda & \lambda+1 & p-2 &0 &0 &\ldots & 0\\
0 &-\lambda &  \lambda+2 & p-3&0 & \ldots & 0\\
\vdots & \vdots & \vdots &\vdots & \vdots & \vdots & \vdots\\
\vdots & \vdots &0& -\lambda &\lambda+p-3 & 2 & 0\\
0  &\ldots &0  & 0& -\lambda &\lambda+p-2 & 1\\
0&  &\ldots &0  & 0 &-\lambda & \lambda+p-1
\end{bmatrix}.\end{equation*}
By a standard linear algebra argument, we have that the rank of the matrix above is $p$ if $\mu\neq 1$ and $p-1$ if $\mu=1$. Thus, in the first case the rank of $\mathrm{Im}(1-(\lambda I+\D)^{-1}-A_\mu)$ is $p$ and hence two arbitrary constant $c,c'$ provide isomorphic solutions. Now, let $\mu=1$ and so the rank of $\mathrm{Im}((\lambda I+\Delta)(1-\A_1) -1)$ is $p-1$. By a standard linear algebra argument, we have that $\mathrm{Im}((\lambda I+\Delta)(1-\A_1) -1)$ is the subspace given by the vectors $(x_1,\ldots,x_p)$ with $x_1+...+x_p=0$ and therefore the vector $(1,0,\ldots,0)$ does not belong to this space. Moreover, the vector $(1,0,\ldots,0)$ does not belong to $\mathrm{Im}(1-(\lambda I+\Delta)^{-1}-\A_\mu)$, because otherwise, being $(1,0,\ldots,0)$ an eigenvector of $(\lambda I+\Delta)$ and $\lambda \neq 0$, it would imply that $(1,0,\ldots,0)\in \mathrm{Im}((\lambda I+\Delta)(1-(\lambda I+\Delta)^{-1}-\A_\mu))$, a contradiction. By a standard calculation, we obtain that there are two isomorphism classes of simple cycle sets of the form $X(\lambda,1,c)$, one is given by $c:=(0,\ldots,0)$ and the other one is given by $c:=(1,0,\ldots,0)$. Hence we obtain $(p-1)\cdotp (p-2)$ simple affine cycle sets for $\mu\neq 1$ and $2\cdotp (p-1)$ for $\mu=1$, therefore the statement follows.
\end{proof}

\section{Applications and examples}
In this section, we collect examples of simple latin cycle sets obtained by our results and we discuss some applications to open questions present in literature.

\medskip

We start by a non-existence result that extend \cite[Lemma 4.13]{bon2019}. In particular, we show that we have no latin cycle sets with cyclic displacement group.

 \begin{theor}\label{latin1}
    Let $X$ be a finite latin cycle set with cyclic displacement group. Then, $|X|=1$.
\end{theor}

\begin{proof}
    Let $n:=|X|$ and suppose that $\dis(X)$ is a cyclic group. At first, suppose that $n$ has an odd prime divisor $p$ and let $m$ be the positive integer such that $n/m=p$. As in \cref{cardpr}, the set $H:=p \dis(X)$ is an ideal of $\dis(X)$ of size $m$ and is a characteristic subgroup of $(\dis(X),+)$ and $(\dis(X),\circ)$, hence it is an ideal of $\mathcal{G}(X)$. Therefore, $H$ induces a congruence $\sim$ on $X$ such that $|X/\sim|=p$, and this implies that there exist a latin cycle set of size $p$, but this contradicts \cite[Theorem 2.13]{etingof1998set}. Now, we consider the case $n=2^t$ for some natural number $t$. The case $t=1$ is not possible by \cite[Theorem 2.13]{etingof1998set}. The case $t=2$ is excluded by an inspection of small cycle sets. The case $t>2$ can be showed as in the first part by \cite[Proposition 5.4]{bachiller2016solutions}.
\end{proof}

By \cite{dietzel2023indecomposable} every indecomposable and irretractable cycle set of size $p^2$ is simple. A complete description of the latin ones follows by \cite[Theorem $5.1$]{dietzel2023indecomposable}, as we show in the following result.

\begin{prop}\label{pquad}
    Let $X$ be a latin cycle set of size $p^2$. Then, $X$ is one of the two indecomposable and irretractable cycle sets of order $4$.
\end{prop}

\begin{proof}
 By \cref{prel3}, $X$ must be indecomposable and irretractable. If $p=2$, we know by \cite{bon2019} that the two indecomposable and irretractable cycle sets of order $4$ are latin. If $p>2$, by \cite[Theorem $5.1$]{dietzel2023indecomposable} $X$ is given by $\mathbb{Z}/p\mathbb{Z}\times \mathbb{Z}/p\mathbb{Z}$ and $(a,x)\cdotp (b,y)=(\alpha b+\alpha x,\alpha y+\alpha \phi(b-a))$, where $\alpha \in \mathbb{Z}/p\mathbb{Z}\setminus \{0\}$ and $\phi$ is a map from $\mathbb{Z}/p\mathbb{Z}$ to $\mathbb{Z}/p\mathbb{Z}$ with $\phi(x)=\phi(-x)$ and $\phi(\alpha x)=\alpha \phi(x)$ for all $x\in \mathbb{Z}/p\mathbb{Z}$.
Let $a,b\in \mathbb{Z}/p\mathbb{Z}$ with $a\neq b$ and set $a':=2b-a$. Since $p>2$ we have $a\neq a'$, but $(a,x)\cdotp (b,y)=(a',x)\cdotp (b,y)  $, a contradiction. 
\end{proof}

\begin{cor}
    Let $X$ be a latin cycle set of size $p^3$. Then, $X$ is simple.
\end{cor}

\begin{proof}
    If $p=2$, one can show by computer calculation that there are no latin cycle sets of size $8$. If $p>2$ and $X$ is not simple, then we obtain a latin cycle set having size in $\{p,p^2\}$, but this contradicts \cite[Theorem 2.13]{etingof1998set} and \cref{pquad}.
\end{proof}

\smallskip
In the following example, we apply \cref{descrpr} to show the simplicity of a right linear cycle set.

\begin{ex}
Using the transitive cycle bases of the left brace $SmallBraces(81,804)$ of \cite{Ve15pack}, we obtain an irretractable cycle set $X$ of size $27$ having the displacement group, isomorphic to the extra-special $3$-group of size $27$, that acts transitively, and hence regularly, on $X$. By \cref{secondp}, this cycle set is right linear over $\dis(X)$. By a long but standard calculation (helping us by GAP), one can write $X$ as a cycle set on a group $(G,\circ)$ isomorphic to $\dis(X)$, where $G:=\{1,...,27\}$ with $1$ as neutral element, given by $x\cdotp y:=\phi(x)\circ \psi(y)$, where
$$\phi:=( 1,12,27,13,18,10,11,20,17, 9,19,21,22,16,26, 5,15, 2, 4, 7,23, 3,25, 6,14,24)$$
and
$$\psi:=( 2,25,16)( 3,14,13)( 4, 8,27)
( 5,23,21)( 6,19,18)( 7,26,22)( 9,20,24)(10,11,17).$$ The unique $\psi$-invariant subgroups are $H_1:=\{1,12,15 \}$ and $H_2:=\{1,5,9,12,15,20,21,23,24\}$. On the other hand, $\phi(1)=12$ and $\phi(15)=2$, hence $H_1$ and $H_2$ are not invariant under $\phi$. Therefore, by \cref{descrpr}, $X$ is a simple cycle set right linear over $(G,\circ)$. 

\end{ex}

In \cite{RaVe21}, the following question was posed.

\begin{ques}(Question $3.16$ of \cite{RaVe21})\label{quest1}
    Let $X$ be a cycle set. Is it true that if some $\sigma_x$ contains a non-trivial cycle coprime with $|X|$, then $X$ is decomposable?
\end{ques}

For some classes of cycle sets, the answer to the previous question is positive (see \cite{castelli2023indecomposable,kanrar2024decomposability}). In \cite{Kanrar2024}, Kanrar and Rump found a family of counterexamples: they are non-simple cycle sets of $2-$power size and the smaller one has size $256$. Below, we show that the construction of smaller counterexamples that are also simple cycle sets is possible. At first, we need a lemma implicitly contained in \cite[Section 5]{rump2020one}.

\begin{lemma}\label{pfixed}
    Let $X$ be an affine cycle set given by $x\cdotp y:=\phi(x)+\psi(y)$ for all $x,y\in X$. Then, $x\cdotp y=y$ if and only if $\phi(x)=(I-\psi)y$. If $(I-\psi)$ is invertible, this is equivalent to $(I-\psi)^{-1}\phi(x)=y$ and in this case every $\sigma_x$ has only one fixed point.  
\end{lemma}

\begin{ex}\label{esbase2}
    Let $p$ be a prime number, with $p>2$, and $X$ be the simple affine cycle set of size $p^p$ provided by \cref{simple Rumples} with  $\lambda=1$, $\mu\neq 1$, and $c=(0,\ldots ,0).$
    Since the determinant of $I-B$ is $1-\mu$, by \cref{pfixed} every left multiplication of $X$ has only one fixed point. Let $x$ be an element of $X$. If every non-trivial cycle of $\sigma_x$ is not coprime with $p$, and hence divisible by $p$, we obtain that number of the fixed points of $\sigma_x$ is divisible by $p$, a contradiction. Therefore, $X$ is a counterexample to \cref{quest1}.
\end{ex}

In \cite{cedo2021constructing}, the following question was posed.

\begin{ques}(Question $7.6$ of \cite{cedo2021constructing})
Does there exist a finite simple cycle set $X$  such that $X = Y\times Z$, the sets $X_y = \{(y, z) \mid  z \in Z\}$, for $y\in  Y$, are blocks of imprimitivity for the 
action of $\mathcal{G}(X)$ on $X$, and $|Z|$ is not a divisor of $|Y|$?
\end{ques}

Below, we provide a family of examples that answer positively to the previous question.

\begin{ex}
Let $p$ be a prime number, with $p>2$, and $X$ be the simple affine cycle set of size $p^p$ obtained by \cref{simple Rumples} with $\lambda=\mu=1$ and $c:=(0,\ldots,0)$. Since $\psi$ has $-(x-1)^p$ as characteristic polynomial, the Jordan canonical form of $\psi$ exists in $\mathbb{Z}/p\mathbb{Z}$ and is an upper triangular matrix. Let $\mathcal{B}:=\{e_1,\ldots,e_p \}$ the basis of $X$ that realizes the required form. Then, the subspace $H$ generated by the set $\{e_1,\ldots,e_{p-1} \}$ is $\psi$-invariant and has dimension $p-1$, moreover $X=H\oplus <e_p>$ and it correspond bijectively to the set $Y\times Z$ given by $Y:=\{0\}\oplus <e_p>$ and $Z:=H\oplus \{0\}$, where the correspondence is given by $h+t e_p\mapsto ((0,te_p),(h,0)) $ for all $h\in H$, $t\in \mathbb{Z}/p\mathbb{Z}$. By \cref{imprblo}, $X$ has a complete blocks system $\{X_y \}_{y\in Y} $, where $X_y:= \{(y, z) \mid  z \in Z\}$, for $y\in  Y$, has size $p^{p-1}$, and with $|Y|=p$. 
\end{ex}

\section{Declarations}

Not applicable.

\section{Ethical Approval }

Not applicable.

\section{Funding}

Not applicable.

\section{Availability of data and materials }

Not applicable.

\bibliographystyle{elsart-num-sort}
\bibliography{Bibliography2}

\end{document}